\documentclass[11pt, twoside, leqno]{article}
\usepackage{amssymb,amsmath,amsfonts,amsthm,mathrsfs}
\usepackage{indentfirst}
\usepackage{graphics}
\usepackage{color}
\usepackage[margin=3cm]{geometry}
\usepackage{enumerate}
\usepackage{amsmath}
\usepackage{bbm}           
\usepackage{bm}
\usepackage{eso-pic,graphicx}
\usepackage{tikz}
\usepackage{cite}
\usepackage{esint}
\usepackage[colorlinks=true, pdfstartview=FitV, linkcolor=cyan, citecolor=magenta, urlcolor=magenta]{hyperref}
\usepackage{pslatex}
\usepackage{rotating}
\usepackage{verbatim}

\makeatletter
\def\Ddots{\mathinner{\mkern1mu\raise\p@
		\vbox{\kern7\p@\hbox{.}}\mkern2mu
		\raise4\p@\hbox{.}\mkern2mu\raise7\p@\hbox{.}\mkern1mu}}
\makeatother

\def\XXint#1#2#3{{\setbox0=\hbox{$#1{#2#3}{\int}$}
		\vcenter{\hbox{$#2#3$}}\kern-.5\wd0}}

\allowdisplaybreaks

\begin{document}
	\def\rn{{\mathbb R^n}}  \def\sn{{\mathbb S^{n-1}}}
	\def\co{{\mathcal C_\Omega}}
	\def\z{{\mathbb Z}}
	\def\nm{{\mathbb (\rn)^m}}
	\def\mm{{\mathbb (\rn)^{m+1}}}
	\def\n{{\mathbb N}}
	\def\cc{{\mathbb C}}
	
	\newtheorem{defn}{Definition}[section]
	\newtheorem{thm}{Theorem}[section]
	\newtheorem{lem}{Lemma}[section]
	\newtheorem{cor}{Corollary}[section]
	\newtheorem{rem}{Remark}[section]
	\newtheorem{pro}{Proposition}[section]
	\newtheorem{que}{Question}
	
	\renewcommand{\theequation}{\arabic{section}.\arabic{equation}}

	\title
{\bf\Large Characterizations for multilinear fractional maximal and integral operators and their commutators on generalized weighted Morrey spaces and applications
\footnotetext{
\newline\indent\hspace{1mm} {\it 2020 Mathematics Subject Classification}: Primary 42B25; Secondary 42B35, 47H60, 47B47.
\newline\indent\hspace{1mm} The second author was in part supported by National Natural Science Foundation of China (No. 12071473). 
The third author was in part supported by National Natural Science Foundation of  China (Grant No. 12271232) and Shandong Jianzhu University Foundation (Grant No. X20075Z0101).
The fourth author was partly supported by National Natural Science Foundation of China (Nos. 12071052 and 12271501)}}
	
	\date{}
	\author{Xi Cen\footnote{Corresponding author, E-mail: xicenmath@gmail.com}, Qianjun He, Xiang Li and Dunyan Yan}
	\maketitle

\begin{center}
\begin{minipage}{13cm}
{\small {\bf Abstract:}\quad
This paper is devoted to studying the  boundedness of multilinear operartors and their commutators on generalized weighted Morrey spaces, which includes multilinear fractional maximal operator and multilinear fractional integral operator.  Moreover, we show that two different characterizations for the boundedness of multilinear fractional maximal operators and their commutators on generalized weighted Morrey spaces under different conditions. As some inportant applications, we give the boundedness of multilinear fractional integral operator on generalized weighted Besov-Morrey spaces and also obtain two embedding theorems as well as apriori estimates for the sub-Laplacian $\mathcal L$.}
\end{minipage}
\end{center}
\section{Introduction}\label{sec1}
Many classical (sub)linear operators in harmonic analysis have their multilinear counterparts which go far beyond the linear cases. The pioneer work on the theory of multilinear operators, initiated by Coifman and Meyer \cite{Coifman1975On,Coifman1978commutators}, multilinear theory attracted much attention, including Christ--Journ\'{e} \cite{Christ1987Polynomial}, Kenig--Stein \cite{Kenig1999Multilinear}, Grafakos--Torrers \cite{Grafakos2002Multilinear,Grafakos2002Maximal} and Demeter--Tao--Thiele \cite{Demeter2008Maximal}. As a multilinearization of  Riesz potential, Kenig and Stein \cite{Kenig1999Multilinear} introduced the following multilinear fractional operators
$$\mathcal{I}_{\alpha,m}(\vec{f})(x):=\int_{(\mathbb{R}^{n})^{m}}\frac{f_{1}(y_{1})\cdots f_{m}(y_{m})}{(|x-y_{1}|+\cdots+|x-y_{m}|)^{mn-\alpha}}d{\vec y},$$
where $\vec{f}=(f_{1},\ldots,f_{m})$ and $0<\alpha<mn$ for  $x\in \mathbb{R}^{n}$ and $m\in\mathbb{N}$. And the authors in \cite{Kenig1999Multilinear} established the bounds on product of Lebesgue spaces. Moen \cite{Moen} studied weighted inequalities for multilinear fractional integral operators, which gives a significant boost to the work for multilinear fractional maximal and integral operators, where the  multilinear fractional maximal operator was defined  by  for 
\begin{equation*}
{{\mathcal M}_\alpha }(\vec f)(x) := \mathop {\sup }\limits_{r > 0} \frac{1}{|B(x,r)|^{m-\frac{\alpha}{n}}}\int_{B(x,r)}\cdots\int_{B(x,r)} | {{f_1}({y_1})}\cdots {{f_m}({y_m})} |d{\vec y},
\end{equation*}
Recently, bounds for multilinear operators on some function spaces have been investigated by many authors. Chen and Xue \cite{Chen-Xue} studied weighted estimates for a class of
multilinear fractional type operators. Burenkov et al. \cite{Gu9} proved the boundedness of the fractional maximal operator in local Morrey-type spaces. Chen and Wu \cite{chen-wu} considered the multiple weighted estimates for commutators of multilinear fractional integral operators. Xue \cite{xue7} studied thg weighted estimates for the iterated commutators of multilinear maximal and fractional type operators. Hu, Li and Wang \cite{Hu-Li-Wang,Hu-Wang} proved the boundedness of multilinear singular integral and fractional integral operators on generalized weighted Morrey spaces. Guliyev \cite{Gu7,Gu8} characterized the boundedness of the fractional maximal and integral operators and their commutators in generalized weighted Morrey spaces on Carnot groups.

Let $b$ be a locally integrable function on $\mathbb{R}^{n}$ and let $T$ be an integral operator. Then we recall the commutator operator defined for a proper function $f$ as
$$T_{b}(f)=bT(f)-T(bf).$$
In the literature, $b$ is also called the symbol function of $T_{b}$.  The pioneer work on $T_{b}$ when $T$ belongs to a class of nonconvolution operators and $b\in {\rm BMO}(\mathbb{R}^{n})$, initiated by Coifman, Rochberg and Weiss \cite{Coifman1976Factorization}, the well-known result of which is a new characterization of ${\rm BMO}(\mathbb{R}^{n})$ space via the boundedness of $T_{b}$. For any ball $B\subset \mathbb{R}^{n}$, ${\rm BMO}(\mathbb{R}^{n})$ is the mean oscallation function space defined via the norm
$$\|b\|_{{\rm BMO}(\mathbb{R}^{n})}:=\sup_{B}\frac{1}{|B|}\int_{B}|b(x)-b_{B}|dx\quad\text{with}\quad b_{B}=\frac{1}{|B|}\int_{B}b(x)dx.$$
The theory of commutators plays an important role in the study of partial differential equations(PDEs).  The well-posedness of
solutions to many elliptic PDEs can be attributed to the corresponding commutator's boundedness for integral operators, which is the topic of this paper. Let $\mathcal{T}$ be a $m$-sublinear operator and  $\vec{b}=(b_1,\ldots,b_m)$ be $m$-folds locally integrable functions, the $m$-sublinear commutator generated by  $\mathcal{T}$ and $\vec{b}$ is defined by
\begin{equation*}
{\mathcal{T}_{\Sigma \vec b}}({f_1}, \ldots ,{f_m})(x) =\sum\limits_{j = 1}^m {\mathcal{T}_{\vec b}^j} (\vec f)(x): = \sum\limits_{j = 1}^m {\mathcal{T}({f_1}, \ldots ,({b_j}(x) - {b_j}){f_j}, \ldots {f_m})(x)},
\end{equation*}
the iterated commutator generated by  $\mathcal{T}$ and $\vec{b}$ is defined by
\begin{equation*}
{\mathcal{T}_{\Pi \vec b}}(\vec f)(x) = \mathcal{T}(({b_1}(x) - {b_1}){f_1}, \ldots ,({b_m}(x) - {b_m}){f_m})(x).
\end{equation*}
In recent decades, many authors have proved the boundedness of the commutators with $\rm{BMO}$ functions of multilinear fractional maximal operator in $\rn$. In this paper, the authors consider to extend the well-known results on generalized weighted Morrey spaces. 

The classical Morrey spaces $L^{p,\lambda}$ were first introduced by Morrey in \cite{Morrey} to study the local behavior of solutions to second order elliptic partial differential equations. After that, the study the boundedness of some classical  operators has attracted much more attention. For example,  Lu, Yang and Zhou \cite{LuYangZhou} studied the sublinear operators with rough kernel on generalized Morrey spaces. Komori and Shirai \cite{Komori} considered the weighted version of Morrey spaces $L^{p,\kappa}(\omega)$ and studied the boundedness of some classical operators such as the Hardy-Littlewood maximal operator and the Calder\'on-Zygmund operator on these spaces. They This greatly promotes the mathematical workers to study the Morrey spaces and singular integral operators.

Since then, Guliyev \cite{Gu1} proved boundedness of higher order commutators of sublinear operators on generalized weighted Morrey spaces. Ismayilova \cite{Ismayilova} studied Calder\'{o}n-Zygmund operators with kernels of Dini's type and their multilinear commutators on generalized Morrey spaces. Lin and Yan \cite{Lin-Yan} considered the multilinear strongly singular Calder\'{o}n-Zygmund operators and commutators on Morrey type spaces. Guliyev \cite{Gu3} proved the boundedness of multilinear Calder\'{o}n-Zygmund operators with kernels of Dini's type and their commutators on generalized local Morrey spaces. He, Wu and Yan \cite{He} consider the bounds for multilinear operators under an integral type condition on Morrey space. Guliyev \cite{Gu2} obtained the boundedness of commutators of multilinear Calder\'{o}n-Zygmund operators with kernels of Dini's type on generalized weighted Morrey spaces and applications. Cen \cite{cen1} proved boundedness of multilinear Littlewood-Paley square operators and their commutators on weighted Morrey spaces.

Let us recall the following definitions of generalized weighted Morrey spaces and generalized local weigthed Morrey spaces. 

Let $1 \le p<\infty$, $\omega$ be a weight function on $\mathbb R^n$, $\varphi$ be a positive measurable function on $\rn \times \left( {0,\infty } \right)$. The generalized weighted Morrey spaces are defined by
\begin{equation*}
{M^{p,\varphi }}(\omega ) = \{ f:{\left\| f \right\|_{{M^{p,\varphi }}(\omega )}}: = \mathop {\sup }\limits_{x \in \rn,r > 0} \varphi {(x,r)^{ - 1}}\omega {(B(x,r))^{ - \frac{1}{p}}}{\left\| f \right\|_{{L^p}(B(x,r),\omega dx)}} < \infty\}.
\end{equation*}
The weak generalized weighted Morrey spaces are defined by
\begin{equation*}
W{M^{p,\varphi }}(\omega ) = \{ f:{\left\| f \right\|_{{M^{p,\varphi }}(\omega )}}: = \mathop {\sup }\limits_{x \in \rn,r > 0} \varphi {(x,r)^{ - 1}}\omega {(B(x,r))^{ - \frac{1}{p}}}{\left\| f \right\|_{W{L^p}(B(x,r),\omega dx)}}< \infty\}.
\end{equation*}
It is easy to check that the following facts:
\begin{enumerate}
	\item[(i)] If $\omega \equiv 1$, then $M^{p,\varphi}(1)=M^{p,\varphi}$ are the generalized Morrey spaces and $WM^{p,\varphi}(1)=WM^{p,\varphi}$ are the weak generalized Morrey spaces;
	\item [(ii)] If $\varphi(x,r) \equiv \omega(B(x,r))^{\frac{\kappa-1}{p}}$, then $M^{p,\varphi}(\omega)=L^{p,\kappa}(\omega)$ is the weighted Morrey spaces;
	\item [(iii)] If  $\varphi(x,r) \equiv v(B(x,r))^{\frac{\kappa}{p}} \omega(B(x,r))^{-\frac{1}{p}}$, then
	$M^{p,\varphi}_{}(\omega)=L^{p,\kappa}_{}(v,\omega)$ is the two weighted Morrey spaces;
	\item [(iv)] If $w\equiv1$ and $\varphi(x,r)=r^{\frac{\lambda-n}{p}}$ with $0<\lambda<n$, then $M^{p,\varphi}(\omega)=L^{p,\lambda}$ is the Morrey spaces and $WM^{p,\varphi}(\omega)=WL^{p,\lambda}$ is the weak Morrey spaces;
	\item [(v)] If   $\varphi(x,r) \equiv \omega(B(x,r))^{-\frac{1}{p}}$, then $M^{p,\varphi}(\omega)=L^{p}(\omega)$ is the weighted Lebesgue spaces.
\end{enumerate}
In \cite{cen2}, the authors consider sufficient conditions for multi-sublinear operators on three kinds of generalized weighted Morrey spaces and show that these conditions can guarantee boundedness of multi-sublinear operators. Inspired by it, we will consider the boundedness of multilinear fractional maximal and integral operators and find the required parameter conditions. We give the characterize the boundedness of the above operators and their commutators on generalized weighted Morrey spaces.

Now, we can formulate the first main results of this paper as:
\begin{center}
	\textbf{1: Multilinear fractional maximal operators}
\end{center}

\begin{thm}\label{FRM}
	Let $m\geq 2$, $1\leq p_i<\infty$, $i=1,2,\ldots,m$ with $\frac{1}{p} = \sum\limits_{i = 1}^m {\frac{1}{{{p_i}}}}$, $\frac{{{\alpha _i}}}{n}{\rm{ = }}\frac{1}{{{p_i}}}{\rm{ - }}\frac{1}{{{q_i}}} \in (0,1)$, $\alpha  = \sum\limits_{i = 1}^m {{\alpha _i}}$, $\frac{1}{q} = \sum\limits_{i = 1}^m {\frac{1}{{{q_i}}}}$, $\vec{\omega} \in {A_{\vec P,q}}$ with ${\omega _i}^{{q_i}} \in {A_\infty }$, $i = 1, \ldots ,m$, and a group of non-negative measurable functions $({{\vec \varphi }_1},{\varphi _2}) = ({\varphi _{11}}, \ldots ,{\varphi _{1m}},{\varphi _2})$ satisfies the condition:
	\begin{equation}\label{con5}
	{\left[ {{{\vec \varphi }_1},{\varphi _2}} \right]_A}: = \mathop {\sup }\limits_{x \in {\rn},r > 0} {\varphi _2}{(x,r)^{{\rm{ - }}1}}\mathop {\sup }\limits_{t > r} \frac{{\mathop {{\rm{essinf}}}\limits_{t < \eta  < \infty } \prod\limits_{i = 1}^m {{\varphi _{1i}}(x,\eta ){{\left\| {{\omega _i}} \right\|}_{{L^{{p_i}}}(B(x,\eta ))}}} }}{{\prod\limits_{i = 1}^m {{{\left\| {{\omega _i}} \right\|}_{{L^{{q_i}}}(B(x,t))}}} }} < \infty.
	\end{equation}
	
	\begin{enumerate}[(i)]
		\item If $\mathop {\min }\limits_{1 \le k \le m} \{ {p_k}\}  > 1$, then $\mathcal{M}_\alpha$ is bounded from ${M^{{p_1},{\varphi _{11}}}}({\omega _1}^{p_1}) \times  \cdots  \times {M^{{p_m},{\varphi _{1m}}}}({\omega _m}^{p_m})$ to ${M^{q,{\varphi _2}}}({u_{\vec \omega }}^q)$, i.e., $\mathcal{M}_\alpha \in B\left( {\prod\limits_{i = 1}^m {{M^{{p_i},{\varphi _{1i}}}}\left( {{\omega _i}^{p_i}} \right)}  \to {M^{q,{\varphi _2}}}\left( {{u_{\vec \omega }}^q} \right)} \right)$.
		\item If $\mathop {\min }\limits_{1 \le k \le m} \{ {p_k}\}  = 1$, then $\mathcal{M}_\alpha$ is bounded from ${M^{{p_1},{\varphi _{11}}}}({\omega _1}^{p_1}) \times  \cdots  \times {M^{{p_m},{\varphi _{1m}}}}({\omega _m}^{p_m})$ to ${WM^{q,{\varphi _2}}}({u_{\vec \omega }}^q)$, i.e., $\mathcal{M}_\alpha \in B\left( {\prod\limits_{i = 1}^m {{M^{{p_i},{\varphi _{1i}}}}\left( {{\omega _i}^{p_i}} \right)}  \to {WM^{q,{\varphi _2}}}\left( {{u_{\vec \omega }}^q} \right)} \right)$.
	\end{enumerate}
\end{thm}

Next, the following boundedness results for the two commutators of multilinear fractional maximal operators ${\mathcal{M}_{\alpha, \prod \vec b }}$ and ${{\mathcal M}_{\alpha ,\sum {\vec b} }}$ on $\prod\limits_{i = 1}^m {{M^{{p_i},{\varphi _{1i}}}}\left( {{\omega _i}^{p_i}} \right)}$ are also valid.

\begin{thm}\label{CFRM}
	Let $m\geq 2$, $1< p_i<\infty$, $i=1,2,\ldots,m$ with $\frac{1}{p} = \sum\limits_{i = 1}^m {\frac{1}{{{p_i}}}}$, $\frac{{{\alpha _i}}}{n}{\rm{ = }}\frac{1}{{{p_i}}}{\rm{ - }}\frac{1}{{{q_i}}} \in (0,1)$, $\alpha  = \sum\limits_{i = 1}^m {{\alpha _i}}$, $\frac{1}{q} = \sum\limits_{i = 1}^m {\frac{1}{{{q_i}}}}$, $\vec{\omega}\in {A_{\vec P,q}}$ with ${\omega _i}^{{q_i}} \in {A_\infty }$, $i = 1, \ldots ,m$, and a group of non-negative measurable functions $({{\vec \varphi }_1},{\varphi _2})$ satisfies the condition:
	\begin{equation}\label{con6}
	{\left[ {{{\vec \varphi }_1},{\varphi _2}} \right]_A}^\prime : = \mathop {\sup }\limits_{x \in {\rn},r > 0} {\varphi _2}{(x,r)^{{\rm{ - }}1}}\mathop {\sup }\limits_{t > r} {\left( {1 + \log \frac{t}{r}} \right)^m}\frac{{\mathop {{\rm{essinf}}}\limits_{t < \eta  < \infty } \prod\limits_{i = 1}^m {{\varphi _{1i}}(x,\eta ){{\left\| {{\omega _i}} \right\|}_{{L^{{p_i}}}(B(x,\eta ))}}} }}{{\prod\limits_{i = 1}^m {{{\left\| {{\omega _i}} \right\|}_{{L^{{q_i}}}(B(x,t))}}} }} < \infty.
	\end{equation}
	If $\vec b \in {\left( {\rm BMO} \right)^m}$, then ${\mathcal{M}_{\alpha, \prod \vec b }}$ is bounded from ${M^{{p_1},{\varphi _{11}}}}(\omega _1^{{p_1}}) \times  \cdots  \times {M^{{p_m},{\varphi _{1m}}}}(\omega _m^{{p_m}})$ to ${M^{q,{\varphi _2}}}({u_{\vec \omega }}^q)$. Moreover, we have 
	$${\left\| {\mathcal{M}_{\alpha, \prod \vec b }} \right\|_{\prod\limits_{i = 1}^m {{M^{{p_i},{\varphi _{1i}}}}\left( {{\omega _i}^{{p_i}}} \right)}  \to {M^{q,{\varphi _2}}}\left( {{u_{\vec \omega }}^q} \right)}} \lesssim \prod\limits_{i = 1}^m {{{\left\| {{b_i}} \right\|}_{\rm BMO}}}$$.
\end{thm}

\begin{thm}\label{CFRM2}
	Let $m\geq 2$, $1< p_i<\infty$, $i=1,2,\ldots,m$ with $\frac{1}{p} = \sum\limits_{i = 1}^m {\frac{1}{{{p_i}}}}$,  $\frac{{{\alpha _i}}}{n}{\rm{ = }}\frac{1}{{{p_i}}}{\rm{ - }}\frac{1}{{{q_i}}} \in (0,1)$, $\alpha  = \sum\limits_{i = 1}^m {{\alpha _i}}$, $\frac{1}{q} = \sum\limits_{i = 1}^m {\frac{1}{{{q_i}}}}$, $\vec{\omega}\in {A_{\vec P,q}}$ with ${\omega _i}^{{q_i}} \in {A_\infty }$, $i = 1, \ldots ,m$, and a group of non-negative measurable functions $({{\vec \varphi }_1},{\varphi _2})$ satisfies the condition:
	\begin{equation}\label{con7}
	{\left[ {{{\vec \varphi }_1},{\varphi _2}} \right]_A}^{\prime \prime }: = \mathop {\sup }\limits_{x \in {\rn},r > 0} {\varphi _2}{(x,r)^{{\rm{ - }}1}}\mathop {\sup }\limits_{t > r} \left( {1 + \log \frac{t}{r}} \right)\frac{{\mathop {{\rm{essinf}}}\limits_{t < \eta  < \infty } \prod\limits_{i = 1}^m {{\varphi _{1i}}(x,\eta ){{\left\| {{\omega _i}} \right\|}_{{L^{{p_i}}}(B(x,\eta ))}}} }}{{\prod\limits_{i = 1}^m {{{\left\| {{\omega _i}} \right\|}_{{L^{{q_i}}}(B(x,t))}}} }} < \infty.
	\end{equation}
	If $\vec b \in {\left( {\rm BMO} \right)^m}$, then ${M_{_{\alpha ,\sum {\vec b} }}^j}$ is bounded from ${M^{{p_1},{\varphi _{11}}}}(\omega _1^{{p_1}}) \times  \cdots  \times {M^{{p_m},{\varphi _{1m}}}}(\omega _m^{{p_m}})$ to ${M^{q,{\varphi _2}}}({u_{\vec \omega }}^q)$. Moreover, we have 
	$${\left\| {M_{_{\alpha ,\sum {\vec b} }}^j} \right\|_{\prod\limits_{i = 1}^m {{M^{{p_i},{\varphi _{1i}}}}\left( {{\omega _i}^{{p_i}}} \right)}  \to {M^{q,{\varphi _2}}}\left( {{u_{\vec \omega }}^q} \right)}} \lesssim {\left\| {{b_j}} \right\|_{\rm BMO}},$$ and 
	$${\left\| {{M_{\alpha ,\sum {\vec b} }}} \right\|_{\prod\limits_{i = 1}^m {{M^{{p_i},{\varphi _{1i}}}}\left( {{\omega _i}^{{p_i}}} \right)}  \to {M^{q,{\varphi _2}}}\left( {{u_{\vec \omega }}^q} \right)}} \lesssim \mathop {\max }\limits_{1 \le i \le m} {\left\| {{b_i}} \right\|_{\rm BMO}}.$$
\end{thm}

The above three theorems show sufficient conditions for the boundedness of $\mathcal{M}_\alpha$, ${\mathcal{M}_{\alpha, \prod \vec b }}$ and ${{\mathcal M}_{\alpha ,\sum {\vec b} }}$. Naturally, it will be a very interesting problem to ask whether  there necessary conditions for the boundedness of the above operators? Thus, in the following theorems, we give a positive answer.

For a weight $\omega$, we denote by $\mathcal{G}_\omega ^p$ the set of all almost decreasing function $\varphi :\rn \times (0,\infty ) \to (0,\infty )$ such that for any ball ${B_0} = B({x_0},{r_0}) \subseteq \rn$, we have
\begin{enumerate}
	\item[\emph{(1)}]
	$\varphi ({x_0},{r_0}) \lesssim \mathop {\inf }\limits_{x \in \rn,0 < r \le {r_0}} \varphi (x,r);$
	\item[\emph{(2)}]
	$\varphi ({x_0},{r_0}){\omega ^p}{({x_0},{r_0})^{\frac{1}{p}}} \lesssim \mathop {\inf }\limits_{x \in \rn,0 < r \le {r_0}} \varphi (x,r){\omega ^p}{(D(x,r))^{\frac{1}{p}}}.$
\end{enumerate}

\begin{thm}\label{SCFRM}
	Let $m\geq 2$, $1\leq p_i<\infty$, $i=1,2,\ldots,m$ with $\frac{1}{p} = \sum\limits_{i = 1}^m {\frac{1}{{{p_i}}}}$,  $\frac{{{\alpha _i}}}{n}{\rm{ = }}\frac{1}{{{p_i}}}{\rm{ - }}\frac{1}{{{q_i}}} \in (0,1)$, $\alpha  = \sum\limits_{i = 1}^m {{\alpha _i}}$, $\frac{1}{q} = \sum\limits_{i = 1}^m {\frac{1}{{{q_i}}}}$, $\vec{\omega} \in {A_{\vec P,q}}$ with ${\omega _i}^{{q_i}} \in {A_\infty }$, $i = 1, \ldots ,m$, ${\varphi _i} \in \mathcal{G}_\omega ^p$, $i = 1, \ldots ,m$, $({{\vec \varphi }_1},{\varphi _2})$ is a group of non-negative measurable functions and ${{\vec \varphi }_1}$ satisfies the condition:
	\begin{equation}\label{cen1}
	{\left[ {{{\vec \varphi }_1}} \right]_A}: = \mathop {\sup }\limits_{x \in {\rn},r > 0} {\left( {{r^\alpha }\prod\limits_{i = 1}^m {{\varphi _{1i}}(x,r)} } \right)^{{\rm{ - }}1}}\mathop {\sup }\limits_{t > r} \frac{{\mathop {{\rm{essinf}}}\limits_{t < \eta  < \infty } \prod\limits_{i = 1}^m {{\varphi _{1i}}(x,\eta ){{\left\| {{\omega _i}} \right\|}_{{L^{{p_i}}}(B(x,\eta ))}}} }}{{\prod\limits_{i = 1}^m {{{\left\| {{\omega _i}} \right\|}_{{L^{{q_i}}}(B(x,t))}}} }} < \infty.
	\end{equation}
	\begin{enumerate}[(i)]
		\item If $\mathop {\min }\limits_{1 \le k \le m} \{ {p_k}\}  > 1$, then
		\begin{equation*}
		{{\mathcal M}_\alpha } \in B\left( {\prod\limits_{i = 1}^m {{M^{{p_i},{\varphi _{1i}}}}\left( {{\omega _i}^{{p_i}}} \right)}  \to {M^{q,{\varphi _2}}}\left( {{u_{\vec \omega }}^q} \right)} \right) \Leftrightarrow \mathop {\sup }\limits_{x \in {\rn},r > 0} {r^\alpha }{\varphi _2}{(x,r)^{ - 1}}\prod\limits_{i = 1}^m {{\varphi _{1i}}(x,r)}  < \infty.
		\end{equation*}
		\item If $\mathop {\min }\limits_{1 \le k \le m} \{ {p_k}\}  = 1$, then 		\begin{equation*}
		{{\mathcal M}_\alpha } \in B\left( {\prod\limits_{i = 1}^m {{M^{{p_i},{\varphi _{1i}}}}\left( {{\omega _i}^{{p_i}}} \right)}  \to {WM^{q,{\varphi _2}}}\left( {{u_{\vec \omega }}^q} \right)} \right) \Leftrightarrow \mathop {\sup }\limits_{x \in {\rn},r > 0} {r^\alpha }{\varphi _2}{(x,r)^{ - 1}}\prod\limits_{i = 1}^m {{\varphi _{1i}}(x,r)}  < \infty.
		\end{equation*}
	\end{enumerate}
\end{thm}

\begin{thm}\label{CCFRM1}
	Let $m\geq 2$, $1< p_i<\infty$, $i=1,2,\ldots,m$ with $\frac{1}{p} = \sum\limits_{i = 1}^m {\frac{1}{{{p_i}}}}$, $\frac{{{\alpha _i}}}{n}{\rm{ = }}\frac{1}{{{p_i}}}{\rm{ - }}\frac{1}{{{q_i}}} \in (0,1)$, $\alpha  = \sum\limits_{i = 1}^m {{\alpha _i}}$, $\frac{1}{q} = \sum\limits_{i = 1}^m {\frac{1}{{{q_i}}}}$, $\vec{\omega}\in {A_{\vec P,q}}$ with ${\omega _i}^{{q_i}} \in {A_\infty }$, $i = 1, \ldots ,m$. Set $\vec b \in {\left( {\rm BMO} \right)^m}$, ${\varphi _i} \in \mathcal{G}_\omega ^p$, $i = 1, \ldots ,m$, $({{\vec \varphi }_1},{\varphi _2})$ is a group of non-negative measurable functions and ${{\vec \varphi }_1}$ satisfies the condition:
	\begin{align}\label{cen2}
	{\left[ {{{\vec \varphi }_1}} \right]_A}^{\prime }: =& \mathop {\sup }\limits_{x \in {\rn},r > 0} {\left( {\frac{{{{\left\| {\prod\limits_{i = 1}^m {({b_j} - {{({b_j})}_{B(x,r)}})} } \right\|}_{{L^q}{(B(x,r),u_{\vec \omega }^q)}}}}}{{{{\left\| {u_{\vec \omega }^{}} \right\|}_{{L^q}(B(x,r))}}\prod\limits_{i = 1}^m {{{\left\| {{b_i}} \right\|}_{\rm BMO}}} }}{r^\alpha }\prod\limits_{i = 1}^m {{\varphi _{1i}}(x,r)} } \right)^{ - 1}}\notag\\
	\times& \mathop {\sup }\limits_{t > r} {\left( {1 + \log \frac{t}{r}} \right)^m}\frac{{\mathop {{\rm{essinf}}}\limits_{t < \eta  < \infty } \prod\limits_{i = 1}^m {{\varphi _{1i}}(x,\eta ){{\left\| {{\omega _i}} \right\|}_{{L^{{p_i}}}(B(x,\eta ))}}} }}{{\prod\limits_{i = 1}^m {{{\left\| {{\omega _i}} \right\|}_{{L^{{q_i}}}(B(x,t))}}} }} < \infty.
	\end{align}
	Then we have
	\begin{align}
	&{\left\| {\mathcal{M}_{\alpha, \prod \vec b }} \right\|_{\prod\limits_{i = 1}^m {{M^{{p_i},{\varphi _{1i}}}}\left( {{\omega _i}^{{p_i}}} \right)}  \to {M^{q,{\varphi _2}}}\left( {{u_{\vec \omega }}^q} \right)}} \lesssim \prod\limits_{i = 1}^m {{{\left\| {{b_i}} \right\|}_{\rm BMO}}}\notag\\
	\Leftrightarrow& \mathop {\sup }\limits_{x \in {\rn},r > 0} {\varphi _2}{(x,r)^{ - 1}}\frac{{{{\left\| {\prod\limits_{i = 1}^m {({b_j} - {{({b_j})}_{B(x,r)}})} } \right\|}_{{L^q}{(B(x,r),u_{\vec \omega }^q)}}}}}{{{{\left\| {u_{\vec \omega }^{}} \right\|}_{{L^q}(B(x,r))}}\prod\limits_{i = 1}^m {{{\left\| {{b_i}} \right\|}_{\rm BMO}}} }}{r^\alpha }\prod\limits_{i = 1}^m {{\varphi _{1i}}(x,r)}  < \infty.
	\end{align}
\end{thm}

\begin{thm}\label{CCFRM2}
	Let $m\geq 2$, $1< p_i<\infty$, $i=1,2,\ldots,m$ with $\frac{1}{p} = \sum\limits_{i = 1}^m {\frac{1}{{{p_i}}}}$, $\frac{{{\alpha _i}}}{n}{\rm{ = }}\frac{1}{{{p_i}}}{\rm{ - }}\frac{1}{{{q_i}}} \in (0,1)$, $\alpha  = \sum\limits_{i = 1}^m {{\alpha _i}}$, $\frac{1}{q} = \sum\limits_{i = 1}^m {\frac{1}{{{q_i}}}}$, $\vec{\omega}\in {A_{\vec P,q}}$ with ${\omega _i}^{{q_i}} \in {A_\infty }$, $i = 1, \ldots ,m$. Set $\vec b \in {\left( {\rm BMO} \right)^m}$, ${\varphi _i} \in \mathcal{G}_\omega ^p$, $i = 1, \ldots ,m$, $({{\vec \varphi }_1},{\varphi _2})$ is a group of non-negative measurable functions and ${{\vec \varphi }_1}$ satisfies the condition:
	\begin{align}\label{cen3}
	{\left[ {{{\vec \varphi }_1}} \right]_A}^{\prime \prime }: =& \mathop {\sup }\limits_{x \in {\rn},r > 0} {\left( {\frac{{{{\left\| {{b_j} - {{({b_j})}_{B(x,r)}}} \right\|}_{{L^q}{(B(x,r),u_{\vec \omega }^q)}}}}}{{{{\left\| {u_{\vec \omega }^{}} \right\|}_{{L^q}(B(x,r))}}{{\left\| {{b_j}} \right\|}_{\rm BMO}}}}{r^\alpha }\prod\limits_{i = 1}^m {{\varphi _{1i}}(x,r)} } \right)^{-1}}\notag\\
	\times& \mathop {\sup }\limits_{t > r} \left( {1 + \log \frac{t}{r}} \right)\frac{{\mathop {{\rm{essinf}}}\limits_{t < \eta  < \infty } \prod\limits_{i = 1}^m {{\varphi _{1i}}(x,\eta ){{\left\| {{\omega _i}} \right\|}_{{L^{{p_i}}}(B(x,\eta ))}}} }}{{\prod\limits_{i = 1}^m {{{\left\| {{\omega _i}} \right\|}_{{L^{{q_i}}}(B(x,t))}}} }} < \infty.
	\end{align}
	Then we have
	\begin{align}
	&{\left\| {M_{_{\alpha ,\sum {\vec b} }}^j} \right\|_{\prod\limits_{i = 1}^m {{M^{{p_i},{\varphi _{1i}}}}\left( {{\omega _i}^{{p_i}}} \right)}  \to {M^{q,{\varphi _2}}}\left( {{u_{\vec \omega }}^q} \right)}} \lesssim {\left\| {{b_j}} \right\|_{\rm BMO}}\notag\\
	\Leftrightarrow& \mathop {\sup }\limits_{x \in {\rn},r > 0} {\varphi _2}{(x,r)^{ - 1}}\frac{{{{\left\| {{b_j} - {{({b_j})}_{B(x,r)}}} \right\|}_{{L^q}{(B(x,r),u_{\vec \omega }^q)}}}}}{{{{\left\| {u_{\vec \omega }^{}} \right\|}_{{L^q}(B(x,r))}}{{\left\| {{b_j}} \right\|}_{\rm BMO}}}}{r^\alpha }\prod\limits_{i = 1}^m {{\varphi _{1i}}(x,r)}  < \infty.
	\end{align}
	Moreover, if the above conditions hold, we have
	$${\left\| {{M_{\alpha ,\sum {\vec b} }}} \right\|_{\prod\limits_{i = 1}^m {{M^{{p_i},{\varphi _{1i}}}}\left( {{\omega _i}^{{p_i}}} \right)}  \to {M^{q,{\varphi _2}}}\left( {{u_{\vec \omega }}^q} \right)}} \lesssim \mathop {\max }\limits_{1 \le i \le m} {\left\| {{b_i}} \right\|_{\rm BMO}}.$$
\end{thm}

\begin{rem}
	In the case $m=1$, the above results were proved in \cite{Gu7}.
\end{rem}

\begin{center}
	\textbf{2: Multilinear fractional integral operators}
\end{center}

In \cite{Hu-Wang}, The authors give the boundedness of multilinear fractional integrals on generalized weighted Morrey spaces, but the conditions they give are related to $\alpha$. In this topics, we give another conditions that do not depend on $\alpha$.
\begin{thm}\label{thm7}
	Let $m\geq 2$, $1\leq p_i<\infty$, $i=1,2,\ldots,m$ with $\frac{1}{p} = \sum\limits_{i = 1}^m {\frac{1}{{{p_i}}}}$, $\frac{{{\alpha _i}}}{n}{\rm{ = }}\frac{1}{{{p_i}}}{\rm{ - }}\frac{1}{{{q_i}}} \in (0,1)$, $\alpha  = \sum\limits_{i = 1}^m {{\alpha _i}}$, $\frac{1}{q} = \sum\limits_{i = 1}^m {\frac{1}{{{q_i}}}}$, $\vec{\omega}=(\omega_1,\ldots,\omega_m) \in {A_{\vec P,q}}$ with ${\omega _i}^{{q_i}} \in {A_\infty }$, $i = 1, \ldots ,m$, and a group of non-negative measurable functions $({{\vec \varphi }_1},{\varphi _2}) = ({\varphi _{11}}, \ldots ,{\varphi _{1m}},{\varphi _2})$ satisfy the condition:
	\begin{equation}\label{con3}
	{\left[ {{{\vec \varphi }_1},{\varphi _2}} \right]_B}: = \mathop {\sup }\limits_{x \in {\rn},r > 0} {\varphi _2}{(x,r)^{{\rm{ - }}1}}\int_r^\infty  {\frac{{\mathop {{\rm{essinf}}}\limits_{t < \eta  < \infty } \prod\limits_{i = 1}^m {{\varphi _{1i}}(x,\eta ){{\left\| {{\omega _i}} \right\|}_{{L^{{p_i}}}(D(x,\eta ))}}} }}{{\prod\limits_{i = 1}^m {{{\left\| {{\omega _i}} \right\|}_{{L^{{q_i}}}(D(x,t))}}} }}\frac{{dt}}{t}}  < \infty.
	\end{equation}
	
	\begin{enumerate}[(i)]
		\item If $\mathop {\min }\limits_{1 \le k \le m} \{ {p_k}\}  > 1$, then ${\mathcal I}_{\alpha,m}$ is bounded from ${M^{{p_1},{\varphi _{11}}}}({\omega _1}^{p_1}) \times  \cdots  \times {M^{{p_m},{\varphi _{1m}}}}({\omega _m}^{p_m})$ to ${M^{q,{\varphi _2}}}({u_{\vec \omega }}^q)$, i.e., ${\mathcal I}_{\alpha,m} \in B\left( {\prod\limits_{i = 1}^m {{M^{{p_i},{\varphi _{1i}}}}\left( {{\omega _i}^{p_i}} \right)}  \to {M^{q,{\varphi _2}}}\left( {{u_{\vec \omega }}^q} \right)} \right)$.
		\item If $\mathop {\min }\limits_{1 \le k \le m} \{ {p_k}\}  = 1$, then ${\mathcal I}_{\alpha,m}$ is bounded from ${M^{{p_1},{\varphi _{11}}}}({\omega _1}^{p_1}) \times  \cdots  \times {M^{{p_m},{\varphi _{1m}}}}({\omega _m}^{p_m})$ to ${WM^{q,{\varphi _2}}}({u_{\vec \omega }}^q)$, i.e., ${\mathcal I}_{\alpha,m}\in B\left( {\prod\limits_{i = 1}^m {{M^{{p_i},{\varphi _{1i}}}}\left( {{\omega _i}^{p_i}} \right)}  \to {WM^{q,{\varphi _2}}}\left( {{u_{\vec \omega }}^q} \right)} \right)$.
	\end{enumerate}
\end{thm}

\begin{thm}\label{CFRI1}
	Let $m\geq 2$, $1< p_i<\infty$, $i=1,2,\ldots,m$ with $\frac{1}{p} = \sum\limits_{i = 1}^m {\frac{1}{{{p_i}}}}$, $\frac{{{\alpha _i}}}{n}{\rm{ = }}\frac{1}{{{p_i}}}{\rm{ - }}\frac{1}{{{q_i}}} \in (0,1)$, $\alpha  = \sum\limits_{i = 1}^m {{\alpha _i}}$, $\frac{1}{q} = \sum\limits_{i = 1}^m {\frac{1}{{{q_i}}}}$, $\vec{\omega}=(\omega_1,\ldots,\omega_m) \in {A_{\vec P,q}}$ with ${\omega _i}^{{q_i}} \in {A_\infty }$, $i = 1, \ldots ,m$, and a group of non-negative measurable functions $({{\vec \varphi }_1},{\varphi _2})$ satisfy the condition:
	\begin{equation}\label{con4}
	\left[ {{{\vec \varphi }_1},{\varphi _2}} \right]_B^\prime : = \mathop {\sup }\limits_{x \in {\rn},r > 0} {\varphi _2}{(x,r)^{{\rm{ - }}1}}\int_r^\infty  {{{\left( {1 + \log \frac{t}{r}} \right)}^m}\frac{{\mathop {{\rm{essinf}}}\limits_{t < \eta  < \infty } \prod\limits_{i = 1}^m {{\varphi _{1i}}(x,\eta ){{\left\| {{\omega _i}} \right\|}_{{L^{{p_i}}}(D(x,\eta ))}}} }}{{\prod\limits_{i = 1}^m {{{\left\| {{\omega _i}} \right\|}_{{L^{{q_i}}}(D(x,t))}}} }}\frac{{dt}}{t}}  < \infty.
	\end{equation}
	
	If $\vec b \in {\left( {\rm BMO} \right)^m}$, then ${{\mathcal I}_{\alpha, m,\prod \vec b }}$ is bounded from ${M^{{p_1},{\varphi _{11}}}}(\omega _1^{{p_1}}) \times  \cdots  \times {M^{{p_m},{\varphi _{1m}}}}(\omega _m^{{p_m}})$ to ${M^{q,{\varphi _2}}}({u_{\vec \omega }}^q)$.
	Moreover, we have 
	$${\left\| {\mathcal{I}_{\alpha,m,\prod \vec b }} \right\|_{\prod\limits_{i = 1}^m {{M^{{p_i},{\varphi _{1i}}}}\left( {{\omega _i}^{{p_i}}} \right)}  \to {M^{q,{\varphi _2}}}\left( {{u_{\vec \omega }}^q} \right)}} \lesssim \prod\limits_{i = 1}^m {{{\left\| {{b_i}} \right\|}_{\rm BMO}}}.$$
\end{thm}

\begin{thm}\label{CFRI2}
	Let $m\geq 2$, $1< p_i<\infty$, $i=1,2,\ldots,m$ with $\frac{1}{p} = \sum\limits_{i = 1}^m {\frac{1}{{{p_i}}}}$, $\frac{{{\alpha _i}}}{n}{\rm{ = }}\frac{1}{{{p_i}}}{\rm{ - }}\frac{1}{{{q_i}}} \in (0,1)$, $\alpha  = \sum\limits_{i = 1}^m {{\alpha _i}}$, $\frac{1}{q} = \sum\limits_{i = 1}^m {\frac{1}{{{q_i}}}}$, $\vec{\omega}=(\omega_1,\ldots,\omega_m) \in {A_{\vec P,q}}$ with ${\omega _i}^{{q_i}} \in {A_\infty }$, $i = 1, \ldots ,m$, and a group of non-negative measurable functions $({{\vec \varphi }_1},{\varphi _2})$ satisfy the condition:
	\begin{equation}
	\left[ {{{\vec \varphi }_1},{\varphi _2}} \right]_B^{\prime \prime}: = \mathop {\sup }\limits_{x \in {\rn},r > 0} {\varphi _2}{(x,r)^{{\rm{ - }}1}}\int_r^\infty  {{{\left( {1 + \log \frac{t}{r}} \right)}}\frac{{\mathop {{\rm{essinf}}}\limits_{t < \eta  < \infty } \prod\limits_{i = 1}^m {{\varphi _{1i}}(x,\eta ){{\left\| {{\omega _i}} \right\|}_{{L^{{p_i}}}(D(x,\eta ))}}} }}{{\prod\limits_{i = 1}^m {{{\left\| {{\omega _i}} \right\|}_{{L^{{q_i}}}(D(x,t))}}} }}\frac{{dt}}{t}}  < \infty.
	\end{equation}
	
	If $\vec b \in {\left( {\rm BMO} \right)^m}$, then ${{\mathcal I}_{_{\alpha ,m,\sum {\vec b} }}^j}$ is bounded from ${M^{{p_1},{\varphi _{11}}}}(\omega _1^{{p_1}}) \times  \cdots  \times {M^{{p_m},{\varphi _{1m}}}}(\omega _m^{{p_m}})$ to ${M^{q,{\varphi _2}}}({u_{\vec \omega }}^q)$.
	Moreover, we have 
	$${\left\| {{\mathcal I}_{_{\alpha ,m,\sum {\vec b} }}^j} \right\|_{\prod\limits_{i = 1}^m {{M^{{p_i},{\varphi _{1i}}}}\left( {{\omega _i}^{{p_i}}} \right)}  \to {M^{q,{\varphi _2}}}\left( {{u_{\vec \omega }}^q} \right)}} \lesssim {\left\| {{b_j}} \right\|_{\rm BMO}},$$ and 
	$${\left\| {{{\mathcal I}_{\alpha ,m,\sum {\vec b} }}} \right\|_{\prod\limits_{i = 1}^m {{M^{{p_i},{\varphi _{1i}}}}\left( {{\omega _i}^{{p_i}}} \right)}  \to {M^{q,{\varphi _2}}}\left( {{u_{\vec \omega }}^q} \right)}} \lesssim \mathop {\max }\limits_{1 \le i \le m} {\left\| {{b_i}} \right\|_{\rm BMO}}.$$
\end{thm}

In this topic, the above three theorems show sufficient conditions for the boundedness of $\mathcal{I}_{\alpha,m}$, ${\mathcal{I}_{\alpha, m,\prod \vec b }}$, ${{\mathcal I}_{\alpha ,m,\sum {\vec b} }}$. Similarly, we give the necessary and sufficient conditions for the boundedness of the above operators as follows whose proofs are omited.
\begin{thm}\label{thm10}
	Let $m\geq 2$, $1\leq p_i<\infty$, $i=1,2,\ldots,m$ with $\frac{1}{p} = \sum\limits_{i = 1}^m {\frac{1}{{{p_i}}}}$, $\frac{{{\alpha _i}}}{n}{\rm{ = }}\frac{1}{{{p_i}}}{\rm{ - }}\frac{1}{{{q_i}}} \in (0,1)$, $\alpha  = \sum\limits_{i = 1}^m {{\alpha _i}}$, $\frac{1}{q} = \sum\limits_{i = 1}^m {\frac{1}{{{q_i}}}}$, $\vec{\omega} \in {A_{\vec P,q}}$ with ${\omega _i}^{{q_i}} \in {A_\infty }$, $i = 1, \ldots ,m$, ${\varphi _i} \in \mathcal{G}_\omega ^p$, $i = 1, \ldots ,m$, $({{\vec \varphi }_1},{\varphi _2})$ is a group of non-negative measurable functions and ${{\vec \varphi }_1}$ satisfy the condition:
	\begin{equation}
	{\left[ {{{\vec \varphi }_1}} \right]_B}: = \mathop {\sup }\limits_{x \in {\rn},r > 0} {\left( {{r^\alpha }\prod\limits_{i = 1}^m {{\varphi _{1i}}(x,r)} } \right)^{{\rm{ - }}1}}\int_r^\infty  {\frac{{\mathop {{\rm{essinf}}}\limits_{t < \eta  < \infty } \prod\limits_{i = 1}^m {{\varphi _{1i}}(x,\eta ){{\left\| {{\omega _i}} \right\|}_{{L^{{p_i}}}(D(x,\eta ))}}} }}{{\prod\limits_{i = 1}^m {{{\left\| {{\omega _i}} \right\|}_{{L^{{q_i}}}(D(x,t))}}} }}\frac{{dt}}{t}}  < \infty.
	\end{equation}
	\begin{enumerate}[(i)]
		\item If $\mathop {\min }\limits_{1 \le k \le m} \{ {p_k}\}  > 1$, then
		\begin{equation*}
		{{\mathcal{I}}_{\alpha,m} } \in B\left( {\prod\limits_{i = 1}^m {{M^{{p_i},{\varphi _{1i}}}}\left( {{\omega _i}^{{p_i}}} \right)}  \to {M^{q,{\varphi _2}}}\left( {{u_{\vec \omega }}^q} \right)} \right) \Leftrightarrow \mathop {\sup }\limits_{x \in {\rn},r > 0} {r^\alpha }{\varphi _2}{(x,r)^{ - 1}}\prod\limits_{i = 1}^m {{\varphi _{1i}}(x,r)}  < \infty.
		\end{equation*}
		\item If $\mathop {\min }\limits_{1 \le k \le m} \{ {p_k}\}  = 1$, then 		\begin{equation*}
		{{\mathcal I}_{\alpha,m} } \in B\left( {\prod\limits_{i = 1}^m {{M^{{p_i},{\varphi _{1i}}}}\left( {{\omega _i}^{{p_i}}} \right)}  \to {WM^{q,{\varphi _2}}}\left( {{u_{\vec \omega }}^q} \right)} \right) \Leftrightarrow \mathop {\sup }\limits_{x \in {\rn},r > 0} {r^\alpha }{\varphi _2}{(x,r)^{ - 1}}\prod\limits_{i = 1}^m {{\varphi _{1i}}(x,r)}  < \infty.
		\end{equation*}
	\end{enumerate}
\end{thm}

\begin{thm}\label{CCFRI1}
	Let $m\geq 2$, $1< p_i<\infty$, $i=1,2,\ldots,m$ with $\frac{1}{p} = \sum\limits_{i = 1}^m {\frac{1}{{{p_i}}}}$, $\frac{{{\alpha _i}}}{n}{\rm{ = }}\frac{1}{{{p_i}}}{\rm{ - }}\frac{1}{{{q_i}}} \in (0,1)$, $\alpha  = \sum\limits_{i = 1}^m {{\alpha _i}}$, $\frac{1}{q} = \sum\limits_{i = 1}^m {\frac{1}{{{q_i}}}}$, $\vec{\omega}\in {A_{\vec P,q}}$ with ${\omega _i}^{{q_i}} \in {A_\infty }$, $i = 1, \cdots ,m$. Set $\vec b \in {\left( {\rm BMO} \right)^m}$, ${\varphi _i} \in \mathcal{G}_\omega ^p$, $i = 1, \ldots ,m$, $({{\vec \varphi }_1},{\varphi _2})$ is a group of non-negative measurable functions and ${{\vec \varphi }_1}$ satisfy the condition:
	\begin{align}
	\left[ {{{\vec \varphi }_1}} \right]_B^{\prime }: =& \mathop {\sup }\limits_{x \in {\rn},r > 0} {\left( {\frac{{{{\left\| {\prod\limits_{i = 1}^m {({b_j} - {{({b_j})}_{B(x,r)}})} } \right\|}_{{L^q}{{(B(x,r),u_{\vec \omega }^q)}}}}}}{{{{\left\| {u_{\vec \omega }^{}} \right\|}_{{L^q}(B(x,r))}}\prod\limits_{i = 1}^m {{{\left\| {{b_i}} \right\|}_{\rm BMO}}} }}{r^\alpha }\prod\limits_{i = 1}^m {{\varphi _{1i}}(x,r)} } \right)^{ - 1}}\notag\\
	\times& \int_r^\infty  {{{\left( {1 + \log \frac{t}{r}} \right)}^m}\frac{{\mathop {{\rm{essinf}}}\limits_{t < \eta  < \infty } \prod\limits_{i = 1}^m {{\varphi _{1i}}(x,\eta ){{\left\| {{\omega _i}} \right\|}_{{L^{{p_i}}}(D(x,\eta ))}}} }}{{\prod\limits_{i = 1}^m {{{\left\| {{\omega _i}} \right\|}_{{L^{{q_i}}}(D(x,t))}}} }}\frac{{dt}}{t}}  < \infty.
	\end{align}
	Then we have
	\begin{align}
	&{\left\| {\mathcal{I}_{\alpha,m,\prod \vec b }} \right\|_{\prod\limits_{i = 1}^m {{M^{{p_i},{\varphi _{1i}}}}\left( {{\omega _i}^{{p_i}}} \right)}  \to {M^{q,{\varphi _2}}}\left( {{u_{\vec \omega }}^q} \right)}} \lesssim \prod\limits_{i = 1}^m {{{\left\| {{b_i}} \right\|}_{\rm BMO}}}\notag\\
	\Leftrightarrow& \mathop {\sup }\limits_{x \in {\rn},r > 0} {\varphi _2}{(x,r)^{ - 1}}\frac{{{{\left\| {\prod\limits_{i = 1}^m {({b_j} - {{({b_j})}_{B(x,r)}})} } \right\|}_{{L^q}{(B(x,r),u_{\vec \omega }^q)}}}}}{{{{\left\| {u_{\vec \omega }^{}} \right\|}_{{L^q}(B(x,r))}}\prod\limits_{i = 1}^m {{{\left\| {{b_i}} \right\|}_{\rm BMO}}} }}{r^\alpha }\prod\limits_{i = 1}^m {{\varphi _{1i}}(x,r)}  < \infty.
	\end{align}
\end{thm}

\begin{thm}\label{CCFIM2}
	Let $m\geq 2$, $1< p_i<\infty$, $i=1,2,\ldots,m$ with $\frac{1}{p} = \sum\limits_{i = 1}^m {\frac{1}{{{p_i}}}}$, $\frac{{{\alpha _i}}}{n}{\rm{ = }}\frac{1}{{{p_i}}}{\rm{ - }}\frac{1}{{{q_i}}} \in (0,1)$, $\alpha  = \sum\limits_{i = 1}^m {{\alpha _i}}$, $\frac{1}{q} = \sum\limits_{i = 1}^m {\frac{1}{{{q_i}}}}$, $\vec{\omega}\in {A_{\vec P,q}}$ with ${\omega _i}^{{q_i}} \in {A_\infty }$, $i = 1, \ldots ,m$. Set $\vec b \in {\left( {\rm BMO} \right)^m}$, ${\varphi _i} \in \mathcal{G}_\omega ^p$, $i = 1, \cdots ,m$, $({{\vec \varphi }_1},{\varphi _2})$ is a group of non-negative measurable functions and ${{\vec \varphi }_1}$ satisfy the condition:
	\begin{align}
	\left[ {{{\vec \varphi }_1}} \right]_B^{\prime \prime }: =& \mathop {\sup }\limits_{x \in {\rn},r > 0} {\left( {\frac{{{{\left\| {{b_j} - {{({b_j})}_{B(x,r)}}} \right\|}_{{L^q}{(B(x,r),u_{\vec \omega }^q)}}}}}{{{{\left\| {u_{\vec \omega }^{}} \right\|}_{{L^q}(B(x,r))}}{{\left\| {{b_j}} \right\|}_{\rm BMO}}}}{r^\alpha }\prod\limits_{i = 1}^m {{\varphi _{1i}}(x,r)} } \right)^{-1}}\notag\\
	\times& \int_r^\infty  {{{\left( {1 + \log \frac{t}{r}} \right)}}\frac{{\mathop {{\rm{essinf}}}\limits_{t < \eta  < \infty } \prod\limits_{i = 1}^m {{\varphi _{1i}}(x,\eta ){{\left\| {{\omega _i}} \right\|}_{{L^{{p_i}}}(D(x,\eta ))}}} }}{{\prod\limits_{i = 1}^m {{{\left\| {{\omega _i}} \right\|}_{{L^{{q_i}}}(D(x,t))}}} }}\frac{{dt}}{t}}  < \infty.
	\end{align}
	Then we have
	\begin{align}
	&{\left\| {{\mathcal I}_{_{\alpha ,m\sum {\vec b} }}^j} \right\|_{\prod\limits_{i = 1}^m {{M^{{p_i},{\varphi _{1i}}}}\left( {{\omega _i}^{{p_i}}} \right)}  \to {M^{q,{\varphi _2}}}\left( {{u_{\vec \omega }}^q} \right)}} \lesssim {\left\| {{b_j}} \right\|_{\rm BMO}}\notag\\
	\Leftrightarrow& \mathop {\sup }\limits_{x \in {\rn},r > 0} {\varphi _2}{(x,r)^{ - 1}}\frac{{{{\left\| {{b_j} - {{({b_j})}_{B(x,r)}}} \right\|}_{{L^q}{(B(x,r),u_{\vec \omega }^q)}}}}}{{{{\left\| {u_{\vec \omega }^{}} \right\|}_{{L^q}(B(x,r))}}{{\left\| {{b_j}} \right\|}_{\rm BMO}}}}{r^\alpha }\prod\limits_{i = 1}^m {{\varphi _{1i}}(x,r)}  < \infty.
	\end{align}
	Moreover, if the above conditions hold, we have
	$${\left\| {{{\mathcal I}_{\alpha ,\sum {\vec b} }}} \right\|_{\prod\limits_{i = 1}^m {{M^{{p_i},{\varphi _{1i}}}}\left( {{\omega _i}^{{p_i}}} \right)}  \to {M^{q,{\varphi _2}}}\left( {{u_{\vec \omega }}^q} \right)}} \lesssim \mathop {\max }\limits_{1 \le i \le m} {\left\| {{b_i}} \right\|_{\rm BMO}}.$$
\end{thm}

\begin{rem}
	In the case $m=1$, the above results in this topic also were proved in \cite{Gu8}.
\end{rem}

\begin{rem}
	Due to ${\mathcal{M}_\alpha }(\vec f)(x) { \lesssim _{n,m,\alpha }} {m^{mn - \alpha }}{\mathcal{I}_{\alpha,m}}(\left| {{f_1}} \right|, \ldots ,\left| {{f_m}} \right|)(x)$, Theorems $\ref{thm7}$, $\ref{CFRI1}$, $\ref{CFRI2}$, $\ref{thm10}$, $\ref{CCFRI1}$, $\ref{CCFIM2}$ are still valid for $\mathcal{M}_\alpha$, ${\mathcal{M}_{\alpha, \prod \vec b }}$ and ${{\mathcal M}_{\alpha ,\sum {\vec b} }}$ respectively. In other words, we obtain two different characterizations for the boundedness of multilinear fractional maximal operators and their commutators on generalized weighted Morrey spaces under different conditions.
\end{rem}

\begin{rem}
	All of the above results still hold for the (weak) generalized local weighted Morrey spaces $M_{{x_0}}^{p,\varphi }(\omega )$ $(WM_{{x_0}}^{p,\varphi }(\omega ))$ (the definitions see Section $\ref{sec2}$), which
	just need to be modified the descriptions where necessary.
\end{rem}

The paper is organized as follows. In Section $\ref{sec2}$, we give some notations and several basic results which will play a crucial role in the sequel. The proof of
Theorems $\ref{FRM}$--$\ref{CCFIM2}$  are presented in Section \ref{sec3}.  In Section $\ref{sec4}$, we will give some applactions to Sobolev embedding theorem and the second-order partial differential operator.

Throughout this paper, we let $C$ denote constants that are independent of the main parameters involved but whose value may differ from line to line. Given a ball $B$ and $\lambda>0$, $\lambda B$ denotes the ball with the same center as $B$ whose radius is $\lambda$ times that of $B$. A weight function $\omega$ is a nonnegative locally integrable function on $\mathbb R^n$ that takes values in $(0,\infty)$ almost everywhere. For a given weight function $\omega$ and a measurable set $E$, we denote the Lebesgue measure of $E$ by $|E|$ and the weighted measure of $E$ by $\omega(E)$, where $\omega(E)=\int_E \omega(x)\,dx$. By $A\lesssim B$ we mean that $A\leq CB$ with some positive constant $C$ independent of appropriate quantities. By $ A \approx B$, we mean that $A\lesssim B$ and $B\lesssim A$.

\section{Some Notation and Basic Results}\label{sec2}

We first recall some standard definitions and notations.  Let $1 \le p<\infty$, $\omega$ be a weight function on $\mathbb R^n$, $\varphi$ be a positive measurable function on $\rn \times \left( {0,\infty } \right)$. The generalized local weighted Morrey spaces are defined by
\begin{equation*}
M_{{x_0}}^{p,\varphi }(\omega ) = \{ f:{\left\| f \right\|_{M_{{x_0}}^{p,\varphi }(\omega )}}: = \mathop {\sup }\limits_{r > 0} \varphi {({x_0},r)^{ - 1}}\omega {(B({x_0},r))^{ - \frac{1}{p}}}{\left\| f \right\|_{{L^p}(B({x_0},r),\omega dx)}} < \infty \}.
\end{equation*}
The weak generalized local weighted Morrey spaces are defined by
\begin{equation*}
WM_{{x_0}}^{p,\varphi }(\omega ) = \{ f:{\left\| f \right\|_{WM_{{x_0}}^{p,\varphi }(\omega )}}: = \mathop {\sup }\limits_{r > 0} \varphi {({x_0},r)^{ - 1}}\omega {(B({x_0},r))^{ - \frac{1}{p}}}{\left\| f \right\|_{{WL^p}(B({x_0},r),\omega dx)}} < \infty \}.
\end{equation*}

The classical $A_p$ weight theory was introduced by Muckenhoupt which studied weighted $L^p$ boundedness of Hardy-Littlewood maximal functions, one can see Chapter 7 in \cite{Gra1}.
\begin{defn}[\cite{Gra1}]
	We say that a weight $\omega\in A_p$,\,$1<p<\infty$, if
	$$\left(\frac1{|B|}\int_B \omega(x)\,dx\right)\left(\frac1{|B|}\int_B \omega(x)^{-\frac{1}{p-1}}\,dx\right)^{p-1}\le C \quad\mbox{for every ball}\; B\subseteq \mathbb R^n,$$ where $C$ is a positive constant which is independent of $B$.

	We say a weight $\omega\in A_1$, if
	$$\frac1{|B|}\int_B \omega(x)\,dx\le C\mathop {ess\inf }\limits_{x \in B} \omega(x)\quad\mbox{for every ball}\;B\subseteq\mathbb R^n.$$
	We denote $${A_\infty } = \bigcup\limits_{1 \le p < \infty } {{A_p}}.$$
\end{defn}

\begin{lem}[\cite{Gra1}]\label{Gra1.1}
	Let $\omega\in {A_\infty }$. Then, for any ball $B$ and $\lambda>1$, there exists a constant ${C_{n,\lambda }}$ such that
	\begin{equation*}
	\omega (\lambda B) \le {C_{n,\lambda }}\omega (B).
	\end{equation*}
\end{lem}

Now let us recall the definitions of the following multiple weights. 
\begin{defn}[\cite{Moen}]
	Let $q>0$, $p_1,\ldots,p_m\in[1,\infty)$ with $1/p=\sum_{k=1}^m 1/{p_k}$. Given $\vec{\omega}=(\omega_1,\ldots,\omega_m)$, set $u_{\vec{\omega}}=\prod_{i=1}^m \omega_i$. We say that $\vec{\omega}$ satisfies the $A_{\vec{P},q}$ condition if it satisfies
	\begin{equation}
	\mathop {\sup }\limits_B {\left( {\frac{1}{{|B|}}\int_B {{u _{\vec \omega }}} {{(x)}^q}{\mkern 1mu} dx} \right)^{1/q}}\prod\limits_{i = 1}^m {{{\left( {\frac{1}{{|B|}}\int_B {{\omega _i}} {{(x)}^{ - {{{p_i}^\prime }}}}{\mkern 1mu} dx} \right)}^{1/{{{p_i}^\prime }}}}}  < \infty.
	\end{equation}
	when $p_i=1,$ ${\left( {\frac{1}{{|B|}}\int_B {{\omega _i}} {{(x)}^{ - {{p_i}^\prime }}}{\mkern 1mu} dx} \right)^{1/{{p_i}^\prime }}}$ is understood as ${(\mathop {\inf }\limits_{x \in B} {\omega _i}(x))^{ - 1}}$.
\end{defn}

\begin{lem}[\cite{Moen}]\label{lem4}
	Let $0 < \alpha  < mn$, $p_1,\ldots,p_m\in[1,\infty)$, $1/p=\sum_{i=1}^m 1/{p_i}$ and $\frac{1}{q} = \frac{1}{p} - \frac{\alpha }{n}>0$. Then $\vec{\omega}\in A_{\vec{P},q}$, then
	\begin{equation}\label{multi2}
	\left\{
	\begin{aligned}
	&{u_{\vec{\omega}}}^q \in A_{mq},\\
	&\omega_i^{-{{p_i}^\prime }}\in A_{m{{p_i}^\prime }},\quad i=1,\ldots,m,
	\end{aligned}\right.
	\end{equation}
	where $u_{\vec{\omega}}=\prod_{i=1}^m \omega_i$ and the condition $\omega_i^{1-{{p_i}^\prime }}\in A_{m{{p_i}^\prime }}$ in the case $p_i=1$ is understood as $\omega_i^{1/m}\in A_1$.
\end{lem}

\begin{lem}\label{cen4}
	Let $m\in \mathbb{N},$ $q_1,\ldots,q_m\in[1,\infty)$ with $1/q=\sum_{i=1}^m 1/{q_i}$. Assume that $\omega _1^{{q_1}}, \ldots ,\omega _m^{{q_m}} \in {A_\infty }$ and $u_{\vec \omega }^q = \prod\limits_{i = 1}^m {\omega _i^q}  \in {A_\infty }$, then for any ball $B$, we have
	\begin{equation*}\label{N1}
	\prod\limits_{i = 1}^m {{{\left\| {{\omega _i}} \right\|}_{{L^{{q_i}}}(B)}}}  \approx {\left\| {{u_{\vec \omega }}} \right\|_{{L^q}(B)}}.
	\end{equation*}
	\begin{proof}[Proof:]
		Using Jensen's inequality and the definition of ${A_\infty }$ which can be found in \cite[p. 12]{Gra1} and \cite[p. 525]{Gra1}, we get
		\begin{equation*}
		\left| B \right|\exp (\frac{1}{{\left| B \right|}}\int_B {\log \omega _i^{{q_i}}} ) \le \omega _i^{{q_i}}(B) \lesssim \left| B \right|\exp (\frac{1}{{\left| B \right|}}\int_B {\log \omega _i^{{q_i}}} ).
		\end{equation*}
		and then we have
		\begin{align*}
		\prod\limits_{i = 1}^m {\omega _i^{{q_i}}{{(B)}^{\frac{q}{{{q_i}}}}}}  \approx \left| B \right|\exp (\frac{1}{{\left| B \right|}}\int_B {\log u_{\vec \omega }^q} ) \approx u_{\vec \omega }^q(B).
		\end{align*}
	\end{proof}
\end{lem}
Next, we introduce some lemmas for ${\rm BMO}$ spaces.
\begin{lem}[\cite{Gra2}]
	For all $p \in [1,\infty )$ and $b \in L_{loc}^1(\rn)$, we have
	\begin{equation*}
	\mathop {\sup }\limits_B {\left(\frac{1}{{\left| B \right|}}\int_B^{} {{{\left| {b(x) - {b_B}} \right|}^p}dx} \right)^{\frac{1}{p}}} \approx {\left\| b \right\|_{\rm BMO}}:=\mathop {\sup }\limits_B \left(\frac{1}{{\left| B \right|}}\int_B^{} {\left| {b(x) - {b_B}} \right|dx} \right).
	\end{equation*}
\end{lem}

\begin{lem}[\cite{Gu1}]\label{Gu1}
	Let $\omega\in A_{\infty}$ and $b \in {\rm BMO}$. Then for any $p \in [1,\infty )$, $r_1, r_2 >0$, we have
	\begin{equation*}
	{\left( {\frac{1}{{\omega \left( {B({x_0},{r_1})} \right)}}\int_{B({x_0},{r_1})} {{{\left| {b(x) - {b_{B({x_0},{r_2})}}} \right|}^p}\omega (x)dx} } \right)^{\frac{1}{p}}} \lesssim {\left\| b \right\|_{\rm BMO}}\left( {1 + \left| {\log \frac{{{r_1}}}{{{r_2}}}} \right|} \right).
	\end{equation*}
\end{lem}

We will use the following statement on the boundedness of the weighted Hardy operators
\begin{equation*}
H_w^{}g(r): = \int_r^\infty  {g(t)w(t)\frac{{dt}}{t}},\qquad  H_w^*g(r): = \int_r^\infty  {{{\left( {1 + \log \frac{t}{r}} \right)}^m}g(t)w(t)\frac{{dt}}{t}}, t>0,
\end{equation*}
where $w$ is a weight. The following Lemmas are also important to prove the main results.

\begin{lem}[\cite{Gu1}]\label{Gu2}
	Let $v_1, v_2$ and $w$ be weights on $\left( {0,\infty } \right)$ and $v_1$ be bounded outside a neighborhood of the origin. The inequality
	\begin{equation}\label{Gu21}
	\mathop {\sup }\limits_{r > 0} {v_2}(r){H_w}g(r) \lesssim \mathop {\sup }\limits_{r > 0} {v_1}(r)g(r)
	\end{equation}
	holds for all non-negative and non-decreasing $g$ on $\left( {0,\infty } \right)$ if and only if
	\begin{equation*}
	B: = \mathop {\sup }\limits_{r > 0} {v_2}(r)\int_r^\infty  {\mathop {{\rm{inf}}}\limits_{t < \eta  < \infty } \left( {{v_1}{{(\eta )}^{ - 1}}} \right)w\left( t \right)dt}  < \infty.
	\end{equation*}
\end{lem}

\begin{lem}[\cite{Gu1}]\label{Gu3}
	Let $v_1, v_2$ and $w$ be weights on $\left( {0,\infty } \right)$ and $v_1$ be bounded outside a neighborhood of the origin. The inequality
	\begin{equation}\label{Gu31}
	\mathop {\sup }\limits_{r > 0} {v_2}(r)H_w^*g(r) \lesssim \mathop {\sup }\limits_{r > 0} {v_1}(r)g(r)
	\end{equation}
	holds for all non-negative and non-decreasing $g$ on $\left( {0,\infty } \right)$ if and only if
	\begin{equation*}
	B: = \mathop {\sup }\limits_{r > 0} {v_2}(r)\int_r^\infty  {{{\left( {1 + \log \frac{t}{r}} \right)}^m}\mathop {{\rm{inf}}}\limits_{t < \eta  < \infty } \left( {{v_1}{{(\eta )}^{ - 1}}} \right)w\left( t \right)dt}  < \infty.
	\end{equation*}
\end{lem}

In the sequel $\mathfrak{M}(0,\infty)$, ${\mathfrak{M}}^+(0,\infty)$ and ${\mathfrak{M}}^+((0,\infty);\uparrow)$ stand for the set of Lebesgue measurable functions on $(0,\infty)$, and its subspaces of nonnegative and nonnegative non-decreasing functions, respectively. We also denote
$$\mathbb{A} = \{ \varphi  \in \mathfrak{M}((0,\infty ); \uparrow ):\mathop {\lim }\limits_{t \to {0^ + }} \varphi  = 0\}.$$

Let $u$ be a continuous and non-negative function on $(0,\infty)$. We define the supremal operator ${{\bar S}_u}$ by
$$({{\bar S}_u}g)(r): = {\left\| {ug} \right\|_{{L^\infty }(r,\infty )}}, r \in (0,\infty ).$$

\begin{lem}[\cite{Gu8}]\label{Gu8}
	Suppose that $v_1$ and $v_2$ are nonnegative measurable functions such that ${\left\| {{v_1}} \right\|_{{L^\infty }(0,t)}} \in (0,\infty )$ for every $t>0$. Let $u$ be a continuous nonnegative function on $\mathbb{R}$.Then the operator ${{\bar S}_u}$ is bounded from ${L^\infty }((0,\infty ),{v_1})$ to ${L^\infty }((0,\infty ),{v_2})$ on the cone $\mathbb{A}$ if and only if
	\begin{equation}\label{Gu81}
	{\left\| {{v_2}{{\bar S}_u}(\left\| {{v_1}} \right\|_{{L^\infty }( \cdot ,\infty )}^{ - 1})} \right\|_{{L^\infty }(0,\infty )}} < \infty.
	\end{equation}
\end{lem}

In order to prove our results, we need to estimate characteristic function on (weak) generalized weighted Morrey spaces
\begin{lem}[\cite{Gu7}]\label{Gu7}
	Let $\omega$ is a weight, ${B_0} = B({x_0},{r_0}) \subseteq \rn$. If $\varphi \in \mathcal{G}_\omega ^p$, we have
	\begin{equation*}
	\varphi {({x_0},{r_0})^{ - 1}} \le {\left\| {{\chi _{{B_0}}}} \right\|_{W{M^{p,\varphi }}\left( {{\omega _{}}^p} \right)}} \le {\left\| {{\chi _{{B_0}}}} \right\|_{{M^{p,\varphi }}\left( {{\omega _{}}^p} \right)}} \lesssim \varphi {({x_0},{r_0})^{ - 1}}.
	\end{equation*}
\end{lem}

To be end, we recall some important weighted estimates of multilinear fractional maximal and integral operators.
\begin{lem}[\cite{Moen,Chen-Xue}]\label{LB1}
	Let $0 < \alpha < mn$, $1 < p_1, \ldots, p_m < \infty$,
	$\frac{1}{p} = \frac{1}{p_1} + \cdots + \frac{1}{p_m}$ and
	$\frac{1}{q} = \frac{1}{p} - \frac{\alpha}{n}$. Then for
	$\vec{\omega}\in A_{\vec{P}, q}$ if and only if either of the two multiple weighted norm inequalities holds.
	\begin{equation}
	{\big\|\mathcal{M}_{\alpha}(\vec{f})\big\|}_{L^{q}({u_{\vec{\omega}}}^{q})} \le C \prod_{i=1}^m{\big\|f_i\big\|}_{L^{p_i}(\omega_i^{p_i})};
	\end{equation}
	\begin{equation}
	{\left\|I_{\alpha,m}(\vec{f})\right\|}_{L^{q}({u_{\vec{\omega}}}^{q})} \le C \prod_{i =1}^m{\big\|f_i\big\|}_{L^{p_i}(\omega_i^{p_i})}.
	\end{equation}
\end{lem}

\begin{lem}[\cite{Moen,Chen-Xue,chen-wu}]\label{LB2}
	Let $0 < \alpha < mn$, $1 \leqslant p_1, \ldots, p_m < \infty$,
	$\frac{1}{p} = \frac{1}{p_1} + \cdots + \frac{1}{p_m}$ and
	$\frac{1}{q} = \frac{1}{p} - \frac{\alpha}{n}$. Then for
	$\vec{\omega}\in A_{\vec{P}, q}$ there is a constant $C > 0$
	independent of $\vec{f}$ such that
	\begin{align}
	{\left\|\mathcal{I}_{\alpha,m}(\vec{f})\right\|}_{L^{q,\infty}({u_{\vec{\omega}}}^{q})} &\le C \prod_{i =1}^m{\big\|f_i\big\|}_{L^{p_i}(\omega_i^{p_i})};\notag\\
	{\left\|\mathcal{M}_{\alpha}(\vec{f})\right\|}_{L^{q,\infty}({u_{\vec{\omega}}}^{q})} &\le C \prod_{i =1}^m{\big\|f_i\big\|}_{L^{p_i}(\omega_i^{p_i})};\notag\\
	{\left\| {{\mathcal{I}_{\alpha ,m,\Pi \vec b}}(\vec f)} \right\|_{{L^q}({u_{\vec \omega }}^q)}} &\le C\prod\limits_{i = 1}^m {{{\left\| {{b_i}} \right\|}_{\rm BMO}}{{\left\| {{f_i}} \right\|}_{{L^{{p_i}}}(\omega _i^{{p_i}})}}};\notag\\
	{\left\| {{\mathcal{M}_{\alpha ,\Pi \vec b}}(\vec f)} \right\|_{{L^q}({u_{\vec \omega }}^q)}} &\le C\prod\limits_{i = 1}^m {{{\left\| {{b_i}} \right\|}_{\rm BMO}}{{\left\| {{f_i}} \right\|}_{{L^{{p_i}}}(\omega _i^{{p_i}})}}};\notag\\
	{\left\| {I_{\alpha,m,\sum {\vec b} }^i(\vec f)} \right\|_{{L^q}({u_{\vec \omega }}^q)}}&\le C{\left\| {{b_i}} \right\|_{\rm BMO}}\prod\limits_{i = 1}^m {{{\left\| {{f_i}} \right\|}_{{L^{{p_i}}}(\omega _i^{{p_i}})}}};\notag\\
	{\left\| {\mathcal{M}_{\alpha ,\sum {\vec b} }^i(\vec f)} \right\|_{{L^q}({u_{\vec \omega }}^q)}} &\le C{\left\| {{b_i}} \right\|_{\rm BMO}}\prod\limits_{i = 1}^m {{{\left\| {{f_i}} \right\|}_{{L^{{p_i}}}(\omega _i^{{p_i}})}}}.
	\end{align}
\end{lem}

\section{Proof of Theorems \ref{FRM}--\ref{CCFIM2}}\label{sec3}

\subsection{Proof of Theorem \ref{FRM}}
\begin{proof}[Proof]
	For any ball $B=B(x_0,r)$, let $f_i=f^0_i+f^{\infty}_i$, where $f^0_i=f_i\chi_{2B}$, $i=1,\ldots,m$ and $\chi_{2B}$ denotes the characteristic function of $2B$. Then, we give the following geometric relations by the triangle inequality
	\begin{enumerate}[(i)]
		\item If $z\in B$, then $B(z,t) \cap B{({x_0},2t)^c} \ne \varnothing  \Rightarrow t > r$.
		\item If $z\in B$ and $t > r$, then $B(z,t) \cap B{({x_0},2r)^c} \subseteq B({x_0},2t)$.
	\end{enumerate}
	For any $z\in B$, $\vec \beta  = \left( {{\beta _1}, \ldots ,{\beta _m}} \right) \ne 0$, we have 
	\begin{align}
	{\mathcal{M}_\alpha }({{\vec f}^{\vec \beta }})(z)=& \mathop {\sup }\limits_{t > 0} \prod\limits_{i = 1}^m {{{\left| {B(z,t)} \right|}^{ - 1 + \frac{\alpha_1 }{{n}}}}\int_{B(z,t)} {\left| {f_i^{{\beta _i}}({y_i})} \right|d{y_i}} }\notag\\
	\lesssim& \mathop {\sup }\limits_{t > r} \prod\limits_{i = 1}^m {{{\left| {B({x_0},2t)} \right|}^{ - 1 + \frac{\alpha_i }{{n}}}}\int_{B({x_0},2t)} {\left| {{f_i}({y_i})} \right|d{y_i}} } \notag\\
	\le& \mathop {\sup }\limits_{t > r} \prod\limits_{i = 1}^m {{{\left| {B({x_0},t)} \right|}^{ - 1 + \frac{\alpha_i }{{n}}}}\int_{B({x_0},t)} {\left| {{f_i}({y_i})} \right|d{y_i}} } \notag\\
	\le& \mathop {\sup }\limits_{t > r} \prod\limits_{i = 1}^m {{{\left| {B({x_0},t)} \right|}^{ - 1 + \frac{\alpha_i }{{n}}}}{{\left\| {{f_i}} \right\|}_{{L^{{p_i}}}(B({x_0},t),\omega _i^{{p_i}})}}{{\left\| {\omega _i^{ - 1}} \right\|}_{{L^{{p_i}^\prime }}(B({x_0},t),\omega _i^{{p_i}})}}}\notag\\
	\lesssim&\mathop {\sup }\limits_{t > r} \prod\limits_{i = 1}^m {\left\| {{\omega _i}} \right\|_{{L^{{q_i}}}(B({x_0},t))}^{ - 1}{{\left\| {{f_i}} \right\|}_{{L^{{p_i}}}(B({x_0},t),\omega _i^{{p_i}})}}},\notag
	\end{align}
	where the last step holds due to the definition of $\vec{\omega} \in {A_{\vec P,q}}$ and Lemma \ref{cen4}. Thus we obtain 
	\begin{align}\label{ieq1}
	{\left\| {{\mathcal{M}_\alpha }({{\vec f}^{\vec \beta }})} \right\|_{{M^{q,{\varphi _{1i}}}}(u_{\vec \omega }^q)}} =& \mathop {\sup }\limits_{{x_0} \in \rn,r > 0} {\varphi _2}{({x_0},r)^{ - 1}}\mathop {\sup }\limits_{t > r} \prod\limits_{i = 1}^m {\left\| {{\omega _i}} \right\|_{{L^{{q_i}}}(B({x_0},t))}^{ - 1}{{\left\| {{f_i}} \right\|}_{{L^{{p_i}}}(B({x_0},t),\omega _i^{{p_i}})}}}\notag\\
	\lesssim&\mathop {\sup }\limits_{{x_0} \in \rn,r > 0} \prod\limits_{i = 1}^m {{\varphi _{1i}}{{({x_0},r)}^{ - 1}}\left\| {{\omega _i}} \right\|_{{L^{{p_i}}}(B({x_0},r))}^{ - 1}{{\left\| {{f_i}} \right\|}_{{L^{{p_i}}}(B({x_0},r),\omega _i^{{p_i}})}}}\notag\\
	\le&\prod\limits_{i = 1}^m {{{\left\| {{f_i}} \right\|}_{{L^{{p_i},{\varphi _{1i}}}}(\omega _i^{{p_i}})}}},
	\end{align}
	where the third inequality holds since we combine (\ref{con5}) with $(\ref{Gu81})$ in Lemma \ref{Gu8}.
	
	For $\vec \beta = 0$, by applying Lemma \ref{LB1}, we have
	\begin{equation*}
	{\left\| {{\mathcal{M}_\alpha }({{\vec f}^0})} \right\|_{{L^q}(B,u_{\vec \omega }^q)}} \lesssim \prod\limits_{i = 1}^m {{{\left\| {f_i^0} \right\|}_{{L^{{p_i}}}(\omega _i^{{p_i}})}}}  = \prod\limits_{i = 1}^m {{{\left\| {{f_i}} \right\|}_{{L^p}(2B,\omega _i^{{p_i}})}}}.
	\end{equation*}
	Thus, we get
	\begin{align}\label{ieq2}
	{\left\| {{\mathcal{M}_\alpha }({{\vec f}^{0}})} \right\|_{{M^{q,{\varphi _{2}}}}(u_{\vec \omega }^q)}} =&\mathop {\sup }\limits_{{x_0} \in \rn,r > 0} {\varphi _2}{({x_0},r)^{ - 1}}u_{\vec \omega }^q{(B({x_0},r))^{ - \frac{1}{q}}}{\left\| {{M_\alpha }({{\vec f}^0})} \right\|_{{L^q}(B,u_{\vec \omega }^q)}}\notag\\
	\lesssim&\mathop {\sup }\limits_{{x_0} \in \rn,r > 0} {\varphi _2}{({x_0},r)^{ - 1}}\prod\limits_{i = 1}^m {\omega _i^{{q_i}}{{(B({x_0},r))}^{ - \frac{1}{{{q_i}}}}}{{\left\| {{f_i}} \right\|}_{{L^{{p_i}}}(B({x_0},2r),\omega _i^{{p_i}})}}}\notag\\
	\lesssim&\mathop {\sup }\limits_{{x_0} \in \rn,r > 0} {\varphi _2}{({x_0},r)^{ - 1}}\prod\limits_{i = 1}^m {\omega _i^{{q_i}}{{(B({x_0},2r))}^{ - \frac{1}{{{q_i}}}}}{{\left\| {{f_i}} \right\|}_{{L^{{p_i}}}(B({x_0},2r),\omega _i^{{p_i}})}}}\notag\\
	\lesssim&\mathop {\sup }\limits_{{x_0} \in {\rn},r > 0} {\varphi _2}{({x_0},r)^{ - 1}}\mathop {\sup }\limits_{t > r} \prod\limits_{i = 1}^m {\left\| {{\omega _i}} \right\|_{{L^{{q_i}}}(B({x_0},t))}^{ - 1}{{\left\| {{f_i}} \right\|}_{{L^{{p_i}}}(B({x_0},t),\omega _i^{{p_i}})}}}\notag\\
	\lesssim&\mathop {\sup }\limits_{{x_0} \in \rn,r > 0} \prod\limits_{i = 1}^m {{\varphi _{1i}}{{({x_0},r)}^{ - 1}}\left\| {{\omega _i}} \right\|_{{L^{{p_i}}}(B({x_0},r))}^{ - 1}{{\left\| {{f_i}} \right\|}_{{L^{{p_i}}}(B({x_0},r),\omega _i^{{p_i}})}}}\notag\\
	\le&\prod\limits_{i = 1}^m {{{\left\| {{f_i}} \right\|}_{{L^{{p_i},{\varphi _{1i}}}}(\omega _i^{{p_i}})}}},
	\end{align}
	where the third step is based on the double property of $A_{p}$ weight, see Lemma \ref{Gra1.1}. Combining (\ref{ieq1}) with (\ref{ieq2}), we can get 
	\begin{equation*}
	{\left\| {{\mathcal{M}_\alpha }(\vec f)} \right\|_{{M^{q,{\varphi _2}}}(u_{\vec \omega }^q)}} \le {\left\| {{M_\alpha }({{\vec f}^0})} \right\|_{{M^{q,{\varphi _2}}}(u_{\vec \omega }^q)}} + \sum\limits_{\vec \beta  \ne 0} {{{\left\| {{M_\alpha }({{\vec f}^{\vec \beta }})} \right\|}_{{M^{q,{\varphi _2}}}(u_{\vec \omega }^q)}}} \lesssim \prod\limits_{i = 1}^m {{{\left\| {{f_i}} \right\|}_{{L^{{p_i},{\varphi _{1i}}}}(\omega _i^{{p_i}})}}}. 
	\end{equation*}
	This finishes the proof of Theorem $\ref{FRM}$.
\end{proof}
\subsection{Proof of Theorem \ref{CFRM}}
For the sake of simplicity, we only consider the case when $m=2$.
\begin{proof}[Proof]
	For any ball $B=B(x_0,r)$, let $f_i=f^0_i+f^{\infty}_i$, where $f^0_i=f_i\chi_{2B}$, $i=1,\ldots,m$ and $\chi_{2B}$ denotes the characteristic function of $2B$. For any $z\in B$, $\vec \beta \ne 0$, we have 
	\begin{align}
	&{\mathcal{M}_{\alpha ,\prod {\vec b} }}({{\vec f}^{\vec \beta }})(z)\notag\\
	=& \mathop {\sup }\limits_{t > 0} \prod\limits_{i = 1}^2 {{{\left| {B(z,t)} \right|}^{ - 1 + \frac{\alpha_i }{{n}}}}\int_{B(z,t)} {\left| {({b_i}({y_i}) - {b_i}(z))f_i^{{\beta _i}}({y_i})} \right|d{y_i}} }\notag\\
	\lesssim&\mathop {\sup }\limits_{t > r} \prod\limits_{i = 1}^2 {{{\left| {B({x_0},2t)} \right|}^{ - 1 + \frac{\alpha_i }{{n}}}}\int_{B({x_0},2t)} {\left| {({b_i}({y_i}) - {b_i}(z))f_i^{{\beta _i}}({y_i})} \right|d{y_i}} }\notag\\
	\le&\mathop {\sup }\limits_{t > r} {\left| {B({x_0},t)} \right|^{ - 2 + \frac{\alpha }{n}}}\prod\limits_{i = 1}^2 {\left| {{b_i}(z) - {\mu _i}} \right|\int_{B({x_0},t)} {\left| {f_i^{}({y_i})} \right|d{y_i}} } \notag\\
	+&\mathop {\sup }\limits_{t > r} {\left| {B({x_0},t)} \right|^{ - 2 + \frac{\alpha }{n}}}\left| {{b_1}(z) - {\mu _1}} \right|\int_{D({x_0},2{c_0}t)} {\left| {f_1^{}({y_1})} \right|d{y_1}} \int_{B({x_0},t)} {\left| {({b_2}({y_2}) - {\mu _2})f_2^{}({y_2})} \right|d{y_2}}\notag\\
	+&\mathop {\sup }\limits_{t > r} {\left| {B({x_0},t)} \right|^{ - 2 + \frac{\alpha }{n}}}\left| {{b_2}(z) - {\mu _2}} \right|\int_{D({x_0},2{c_0}t)} {\left| {f_2^{}({y_2})} \right|d{y_2}} \int_{B({x_0},t)} {\left| {({b_1}({y_1}) - {\mu _1})f_1^{}({y_1})} \right|d{y_1}}\notag\\
	+&\mathop {\sup }\limits_{t > r} {\left| {B({x_0},t)} \right|^{ - 2 + \frac{\alpha }{n}}}\prod\limits_{i = 1}^2 {\int_{B({x_0},t)} {\left| {({b_i}({y_i}) - {\mu _i})f_i^{}({y_i})} \right|d{y_i}} }\notag\\
	:=&{N_1}(z) + {N_2}(z) + {N_3}(z) + {N_4}(z).
	\end{align}
	
	For ${N_1}$, by H\"older's inequality, $\vec{\omega} \in {A_{\vec P,q}}$ and Lemma \ref{cen4}, we can see 
	\begin{align}
	{N_1}(z)\notag
	\le& (\prod\limits_{i = 1}^2 {\left| {{b_i}(z) - {\mu _i}} \right|} )\mathop {\sup }\limits_{t > r} {\left| {B({x_0},t)} \right|^{ - 2 + \frac{\alpha }{n}}}\prod\limits_{i = 1}^2 {{{\left\| {{f_i}} \right\|}_{{L^{{P_i}}}(B({x_0},t),\omega _i^{{p_i}})}}{{\left\| {\omega _i^{ - 1}} \right\|}_{{L^{{p_i}^\prime }}(B({x_0},t),\omega _i^{{p_i}})}}}\notag\\
	\lesssim&(\prod\limits_{i = 1}^2 {\left| {{b_i}(z) - {\mu _i}} \right|} )\mathop {\sup }\limits_{t > r} \prod\limits_{i = 1}^m {\left\| {{\omega _i}} \right\|_{{L^{{q_i}}}(B({x_0},t))}^{ - 1}{{\left\| {{f_i}} \right\|}_{{L^{{p_i}}}(B({x_0},t),\omega _i^{{p_i}})}}},\notag
	\end{align}
	and then by Lemma \ref{Gu1}, we obtain
	\begin{align}
	&{\left\| {{N_1}} \right\|_{{L^q}(B,u_{\vec \omega }^q)}}\notag\\
	\le&\prod\limits_{i = 1}^2 {{{\left\| {{b_i}(z) - {\mu _i}} \right\|}_{{L^{2q}}(B,u_{\vec \omega }^q)}}} \mathop {\sup }\limits_{t > r} \prod\limits_{i = 1}^m {\left\| {{\omega _i}} \right\|_{{L^{{q_i}}}(B({x_0},t))}^{ - 1}{{\left\| {{f_i}} \right\|}_{{L^{{p_i}}}(B({x_0},t),\omega _i^{{p_i}})}}}\notag\\
	\lesssim&u_{\vec \omega }^q{(B)^{\frac{1}{q}}}\prod\limits_{i = 1}^2 {{{\left\| {{b_i}} \right\|}_{BMO}}} \mathop {\sup }\limits_{t > r} \prod\limits_{i = 1}^m {\left\| {{\omega _i}} \right\|_{{L^{{q_i}}}(B({x_0},t))}^{ - 1}{{\left\| {{f_i}} \right\|}_{{L^{{p_i}}}(B({x_0},t),\omega _i^{{p_i}})}}}.
	\end{align}
	
	For ${N_2}$, by H\"older's inequality, $\vec{\omega} \in {A_{\vec P,q}}$ and Lemma \ref{cen4}, we can see 
	\begin{align*}
	&{N_2}(z)\notag\\
	\le&\left| {{b_1}(z) - {\mu _1}} \right|\mathop {\sup }\limits_{t > r} {\left| {B({x_0},2t)} \right|^{ - 2 + \frac{\alpha }{n}}}{\left\| {\omega _1^{ - 1}} \right\|_{{L^{{p_1}^\prime }}(B({x_0},t))}}{\left\| {{b_2} - {\mu _2}} \right\|_{{L^{{p_i}^\prime }}(B({x_0},t),\omega _i^{ - {p_i}^\prime })}}\prod\limits_{i = 1}^2 {{{\left\| {{f_i}} \right\|}_{{L^{{p_i}}}(B({x_0},t),\omega _i^{{p_i}})}}}\notag\\
	\lesssim&\left| {{b_1}(z) - {\mu _1}} \right|{\left\| {{b_2}} \right\|_{\rm BMO}}\mathop {\sup }\limits_{t > r} (1 + \log \frac{t}{r})\prod\limits_{i = 1}^2 {{{\left\| {{f_i}} \right\|}_{{L^{{p_i}}}(B({x_0},t),\omega _i^{{p_i}})}}\left\| {{\omega _i}} \right\|_{{L^{{q_i}}}(B({x_0},t))}^{ - 1}},
	\end{align*}
	and then by Lemma \ref{Gu1}, we get
	\begin{equation}
	{\left\| {{N_2}} \right\|_{{L^q}(B,u_{\vec \omega }^q)}} \le u_{\vec \omega }^q{(B)^{\frac{1}{q}}}(\prod\limits_{i = 1}^2 {{{\left\| {{b_i}} \right\|}_{\rm BMO}}} )\mathop {\sup }\limits_{t > r} (1 + \log \frac{t}{r})\prod\limits_{i = 1}^2 {{{\left\| {{f_i}} \right\|}_{{L^{{p_i}}}(B({x_0},t),\omega _i^{{p_i}})}}\left\| {{\omega _i}} \right\|_{{L^{{q_i}}}(B({x_0},t))}^{ - 1}}.
	\end{equation}
	
	For ${N_4}$, similar to estimate $N_{1}(z)$ and $N_{2}(z)$, we have
	\begin{equation}
	{\left\| {{N_4}} \right\|_{{L^q}(B,u_{\vec \omega }^q)}} \lesssim u_{\vec \omega }^q{(B)^{\frac{1}{q}}}(\prod\limits_{i = 1}^2 {{{\left\| {{b_i}} \right\|}_{\rm BMO}}} )\mathop {\sup }\limits_{t > r} {\left( {1 + \log \frac{t}{r}} \right)^2}\prod\limits_{i = 1}^2 {{{\left\| {{f_i}} \right\|}_{{L^{{p_i}}}(B({x_0},t),\omega _i^{{p_i}})}}\left\| {{\omega _i}} \right\|_{{L^{{q_i}}}(B({x_0},t))}^{ - 1}}.
	\end{equation}
	
	Since the estimate of $N_3$ is similar to estimates of $N_2$, combining with above estimates, we can obtain
	\begin{align}\label{ieq3}
	&{\left\| {{\mathcal{M}_{\alpha ,\prod {\vec b} }}({{\vec f}^{\vec \beta }})} \right\|_{{M^{q,{\varphi _{1i}}}}(u_{\vec \omega }^q)}}\notag\\
	\lesssim&\prod\limits_{i = 1}^2 {{{\left\| {{b_i}} \right\|}_{\rm BMO}}} \mathop {\sup }\limits_{{x_0} \in \rn,r > 0} {\varphi _2}{({x_0},r)^{ - 1}}\mathop {\sup }\limits_{t > r} {\left( {1 + \log \frac{t}{r}} \right)^2}\prod\limits_{i = 1}^2 {\left\| {{\omega _i}} \right\|_{{L^{{q_i}}}(B({x_0},t))}^{ - 1}{{\left\| {{f_i}} \right\|}_{{L^{{p_i}}}(B({x_0},t),\omega _i^{{p_i}})}}}\notag\\
	\lesssim&\prod\limits_{i = 1}^2 {{{\left\| {{b_i}} \right\|}_{\rm BMO}}} \mathop {\sup }\limits_{{x_0} \in \rn,r > 0} \prod\limits_{i = 1}^m {{\varphi _{1i}}{{({x_0},r)}^{ - 1}}\left\| {{\omega _i}} \right\|_{{L^{{p_i}}}(B({x_0},t))}^{ - 1}{{\left\| {{f_i}} \right\|}_{{L^{{p_i}}}(B({x_0},t),\omega _i^{{p_i}})}}}\notag\\
	\le&\prod\limits_{i = 1}^2 {{{\left\| {{b_i}} \right\|}_{\rm BMO}}{{\left\| {{f_i}} \right\|}_{{L^{{p_i},{\varphi _{1i}}}}(\omega _i^{{p_i}})}}}.
	\end{align}
	
	For $\vec \beta = 0$, by applying Lemma \ref{LB2}, we have
	\begin{equation*}
	{\left\| {{\mathcal{M}_{\alpha ,\prod {\vec b} }}({{\vec f}^0})} \right\|_{{L^q}(B,u_{\vec \omega }^q)}} \lesssim \prod\limits_{i = 1}^m {{{\left\| {{b_i}} \right\|}_{\rm BMO}}{{\left\| {f_i^0} \right\|}_{{L^{{p_i}}}(\omega _i^{{p_i}})}}}  = \prod\limits_{i = 1}^m {{{\left\| {{b_i}} \right\|}_{\rm BMO}}{{\left\| {{f_i}} \right\|}_{{L^p}(2B,\omega _i^{{p_i}})}}}.
	\end{equation*}

	Thus, we give the following boundedness results
	\begin{align}\label{ieq4}
	&{\left\| {{\mathcal{M}_{\alpha ,\prod {\vec b} }}({{\vec f}^{0}})} \right\|_{{M^{q,{\varphi _{1i}}}}(u_{\vec \omega }^q)}}\notag\\
	=&\mathop {\sup }\limits_{{x_0} \in \rn,r > 0} {\varphi _2}{({x_0},r)^{ - 1}}u_{\vec \omega }^q{(B({x_0},r))^{ - \frac{1}{q}}}{\left\| {{M_\alpha }({{\vec f}^0})} \right\|_{{L^q}(B,u_{\vec \omega }^q)}}\notag\\
	\lesssim&\prod\limits_{i = 1}^m {{{\left\| {{b_i}} \right\|}_{\rm BMO}}}\mathop {\sup }\limits_{{x_0} \in \rn,r > 0} {\varphi _2}{({x_0},r)^{ - 1}}\prod\limits_{i = 1}^m {\omega _i^{{q_i}}{{(B({x_0},r))}^{ - \frac{1}{{{q_i}}}}}{{\left\| {{f_i}} \right\|}_{{L^{{p_i}}}(B({x_0},2r),\omega _i^{{p_i}})}}}\notag\\
	\lesssim&\prod\limits_{i = 1}^m {{{\left\| {{b_i}} \right\|}_{\rm BMO}}}\mathop {\sup }\limits_{{x_0} \in \rn,r > 0} {\varphi _2}{({x_0},r)^{ - 1}}\prod\limits_{i = 1}^m {\omega _i^{{q_i}}{{(B({x_0},2r))}^{ - \frac{1}{{{q_i}}}}}{{\left\| {{f_i}} \right\|}_{{L^{{p_i}}}(B({x_0},2r),\omega _i^{{p_i}})}}}\notag\\
	\lesssim&\prod\limits_{i = 1}^m {{{\left\| {{b_i}} \right\|}_{\rm BMO}}}\mathop {\sup }\limits_{{x_0} \in {\rn},r > 0} {\varphi _2}{({x_0},r)^{ - 1}}\mathop {\sup }\limits_{t > r} \prod\limits_{i = 1}^m {\left\| {{\omega _i}} \right\|_{{L^{{q_i}}}(B({x_0},t))}^{ - 1}{{\left\| {{f_i}} \right\|}_{{L^{{p_i}}}(B({x_0},t),\omega _i^{{p_i}})}}}\notag\\
	\lesssim&\prod\limits_{i = 1}^m {{{\left\| {{b_i}} \right\|}_{\rm BMO}}}\mathop {\sup }\limits_{{x_0} \in \rn,r > 0} \prod\limits_{i = 1}^m {{\varphi _{1i}}{{({x_0},r)}^{ - 1}}\left\| {{\omega _i}} \right\|_{{L^{{p_i}}}(B({x_0},r))}^{ - 1}{{\left\| {{f_i}} \right\|}_{{L^{{p_i}}}(B({x_0},r),\omega _i^{{p_i}})}}}\notag\\
	\le&\prod\limits_{i = 1}^m {{{\left\| {{b_i}} \right\|}_{\rm BMO}}}\prod\limits_{i = 1}^m {{{\left\| {{f_i}} \right\|}_{{L^{{p_i},{\varphi _{1i}}}}(\omega _i^{{p_i}})}}},
	\end{align}
	where we used Lemmas \ref{cen4}, \ref{Gra1.1}, \ref{Gu8} in the second, third, and fifth step, respectively.
	
	Then, we obtain
	\begin{align*}
	{\left\| {{\mathcal{M}_{\alpha ,\prod {\vec b} }}(\vec f)} \right\|_{{M^{q,{\varphi _2}}}(u_{\vec \omega }^q)}} \le& {\left\| {{\mathcal{M}_{\alpha ,\prod {\vec b} }}({{\vec f}^0})} \right\|_{{M^{q,{\varphi _2}}}(u_{\vec \omega }^q)}} + \sum\limits_{\vec \beta  \ne 0} {{{\left\| {{\mathcal{M}_{\alpha ,\prod {\vec b} }}({{\vec f}^{\vec \beta }})} \right\|}_{{M^{q,{\varphi _2}}}(u_{\vec \omega }^q)}}}\notag\\ 
	\lesssim& \prod\limits_{i = 1}^m {{{\left\| {{b_i}} \right\|}_{\rm BMO}}{{\left\| {{f_i}} \right\|}_{{L^{{p_i},{\varphi _{1i}}}}(\omega _i^{{p_i}})}}}.
	\end{align*}
	This completes the proof Theorem \ref{CFRM}.
\end{proof}

\subsection{Proof of Theorems \ref{SCFRM}--\ref{CCFRM2}}
We just need to prove Theorem \ref{CCFRM1} and the remaining proofs are similar.
\begin{proof}[Proof of Theorem \ref{CCFRM1}]
	Combined with Theorem \ref{CFRM}, we only need to prove necessity.
	Suppose that $${\left\| {\mathcal{M}_{\alpha, \prod \vec b }} \right\|_{\prod\limits_{i = 1}^m {{M^{{p_i},{\varphi _{1i}}}}\left( {{\omega _i}^{{p_i}}} \right)}  \to {M^{q,{\varphi _2}}}\left( {{u_{\vec \omega }}} \right)}} \lesssim \prod\limits_{i = 1}^m {{{\left\| {{b_i}} \right\|}_{\rm BMO}}}.$$
	For any $B=B(x_0,r)$, by simple calculations, we can obtain
	\begin{align*}
	{r^\alpha }\prod\limits_{i = 1}^m {\left| {{b_i}(x) - {{({b_i})}_B}} \right|}  \lesssim {M_{\alpha ,\prod {\vec b} }}({\vec \chi _B})(x).
	\end{align*}
	where ${{\vec \chi }_B} = ({\chi _B}, \ldots ,{\chi _B})$.
	
	Combining the above boundedness and Lemma \ref{Gu7}, we get
	\begin{align*}
	\frac{{{r^\alpha }}}{{{{\left\| {{u_{\vec \omega }}} \right\|}_{{L^q}(B)}}}}{\left\| {\prod\limits_{i = 1}^m {\left| {{b_i}(x) - {{({b_i})}_B}} \right|} } \right\|_{{L^q}(B,{u_{\vec \omega }}^q)}} \lesssim& {\varphi _2}({x_0},r){\left\| {{M_{\alpha ,\prod {\vec b} }}({{\vec \chi }_B})} \right\|_{{M^{q,{\varphi _2}}}({u_{\vec \omega }}^q)}}\notag\\ 
	\lesssim& {\varphi _2}({x_0},r)\prod\limits_{i = 1}^m {{{\left\| {{b_i}} \right\|}_{\rm BMO}}} {\left\| {{\chi _B}} \right\|_{{M^{{p_i},{\varphi _{1i}}}}(\omega _i^{{p_i}})}}\notag\\ 
	\approx& {\varphi _2}({x_0},r)\prod\limits_{i = 1}^m {{{\left\| {{b_i}} \right\|}_{\rm BMO}}} {\varphi _{1i}}{({x_0},r)^{ - 1}}.
	\end{align*}
	Thus, we get
	$$\mathop {\sup }\limits_{x \in {\rn},r > 0} {\varphi _2}{(x,r)^{ - 1}}\frac{{{{\left\| {\prod\limits_{i = 1}^m {({b_j} - {{({b_j})}_{B(x,r)}})} } \right\|}_{{L^q}(u_{\vec \omega }^q,B(x,r))}}}}{{{{\left\| {u_{\vec \omega }^{}} \right\|}_{{L^q}(B(x,r))}}\prod\limits_{i = 1}^m {{{\left\| {{b_i}} \right\|}_{\rm BMO}}} }}{r^\alpha }\prod\limits_{i = 1}^m {{\varphi _{1i}}(x,r)}  < \infty.$$
	Hence, we obtain the desired result.
\end{proof}

\subsection{Proof of Theorem \ref{thm7}}
The proof of (ii) is similar to the proof of (i), so we merely conside the proof of (i).
\begin{proof}[Proof]
	For any ${f_i} \in {M^{{p_i},{\varphi _{1i}}}}\left( {{{\omega _i}^{{p_i}}}} \right)$, let $f_i=f^0_i+f^{\infty}_i$, where $f^0_i=f_i\chi_{2B}$, $i=1,\ldots,m$ and $\chi_{2B}$ denotes the characteristic function of $2B$. For ${f_i} \in {M^{{p_i},{\varphi _{1i}}}}\left( {{{\omega _i}^{{p_i}}}} \right)$, we define the  
	\begin{align*}
	{{\cal I}_{\alpha ,m}}(\vec f)(x): = {{\cal I}_{\alpha ,m}}(f_1^0, \ldots ,f_m^0)(x){\rm{ + }}\sum\limits_{({\beta _1}, \ldots ,{\beta _m}) \ne 0} {{{\cal I}_{\alpha ,m}}(f_1^{{\beta _1}}, \ldots ,f_m^{{\beta _m}})(x)},
	\end{align*}
	where we can claim that ${{\cal I}_{\alpha ,m}}$ is well-defined on $\prod\limits_{i = 1}^m {{M^{{p_i},{\varphi _{1i}}}}\left( {{\omega _i}^{{p_i}}} \right)}$, see also Lemma 4.1 in \cite{Gu2}.
	
	Firstly, we use the piecewise integration technique to perform the following estimates. For any $x \in B=B(x_0,r)$ and $1 \le l \le m$, we have
	\begin{align}\label{ieq6}
	&{{\mathcal I}_{\alpha,m}}(f_1^\infty , \ldots ,f_\ell ^\infty ,f_{\ell  + 1}^0, \ldots ,f_m^0)(x)\notag\\ 
	\lesssim& {\int _{{{({\rn})}^l }\backslash {{(2B)}^l }}}\int_{{{(2B)}^{m -l }}} {\frac{{|{f_1}({y_1}) \cdots {f_m}({y_m})|}}{{{{(|x - {y_1}| +  \cdots  + |x - {y_m}|)}^{mn - \alpha }}}}} d{y_1} \cdots d{y_m}\notag\\ 
	\lesssim& (\prod\limits_{i =l  + 1}^m {\int_{2B} | } {f_i}({y_i})|{\mkern 1mu} d{y_i}) \times \sum\limits_{j = 1}^\infty  {\frac{1}{{|{2^{j + 1}}B{|^{m - \frac{\alpha }{n}}}}}} {\int _{{{({2^{j + 1}}B)}^l }\backslash {{({2^j}B)}^l }}}|{f_1}({y_1}) \cdots {f_l }({y_l })|{\mkern 1mu} d{y_1} \cdots d{y_l }\notag\\ 
	\lesssim &\sum\limits_{j = 1}^\infty  {\prod\limits_{i = 1}^m {\frac{1}{{{{\left| {{2^{j + 1}}B} \right|}^{1 - \frac{\alpha_i }{{n}}}}}}} \int_{{2^{j + 1}}B} {\left| {{f_i}({y_i})} \right|d{y_i}} }\\
	\lesssim&\sum\limits_{j = 1}^\infty  {{{\left( {{2^{j + 1}}r} \right)}^{\alpha  - mn}}\prod\limits_{i = 1}^m {{{\left\| {{f_i}} \right\|}_{{L^{{p_i}}}({2^{j + 1}}B,\omega _i^{{p_i}})}}{{\left\| {{\omega _i}^{ - 1}} \right\|}_{{L^{{p_i}^\prime }}({2^{j + 1}}B)}}} }\notag\\
	\le&\sum\limits_{j = 1}^\infty  {{{\left( {{2^{j + 1}}r} \right)}^{\alpha  - mn - 1}}\int_{{2^{j + 1}}r}^{{2^{j + 2}}r} {\prod\limits_{i = 1}^m {{{\left\| {{f_i}} \right\|}_{{L^{{p_i}}}(D({x_0},t),\omega _i^{{p_i}})}}{{\left\| {{\omega _i}^{ - 1}} \right\|}_{{L^{{p_i}^\prime }}(B({x_0},t))}}} dt} }\notag\\
	\lesssim&\int_{2r}^\infty  {\prod\limits_{i = 1}^m {{{\left\| {{f_i}} \right\|}_{{L^{{p_i}}}(B({x_0},t),\omega _i^{{p_i}})}}{{\left\| {{\omega _i}^{ - 1}} \right\|}_{{L^{{p_i}^\prime }}(B({x_0},t))}}} \frac{{dt}}{{{t^{mn - \alpha  + 1}}}}} \notag\\
	\lesssim&\int_{2r}^\infty  {\prod\limits_{i = 1}^m {{{\left\| {{f_i}} \right\|}_{{L^{{p_i}}}(B({x_0},t),\omega _i^{{p_i}})}}\left\| {{\omega _i}} \right\|_{{L^{{q_i}}}(B({x_0},t))}^{ - 1}} \frac{{dt}}{t}},
	\end{align}
	where the last step is based on the definition of $\vec{\omega} \in {A_{\vec P,q}}$ and Lemma \ref{cen4}. 
	
	On the other hand, by using H\"{o}lder's inequality, we have 
	\begin{align}\label{ieq5}
	&{\left\| {{\mathcal I}_{\alpha,m}(f_1^0, \ldots ,f_m^0)} \right\|_{{L^q}(B,{u_{\vec \omega }}^q)}}\notag\\
	\lesssim & \prod\limits_{i = 1}^m {{{\left\| {f_i^0} \right\|}_{{L^{{p_i}}}({\omega _i}^{{p_i}})}}}\notag\\
	\approx& {\left| B \right|^{m - \frac{\alpha }{n}}}\int_{2r}^\infty  {\frac{{dt}}{{{t^{mn - \alpha  + 1}}}}} \prod\limits_{i = 1}^m {{{\left\| {{f_i}} \right\|}_{{L^{{p_i}}}(2B,{\omega _i}^{{p_i}})}}}\notag\\
	\le & {\left| B \right|^{m - \frac{\alpha }{n}}}\int_{2r}^\infty  {\prod\limits_{i = 1}^m {{{\left\| {{f_i}} \right\|}_{{L^{{p_i}}}(B\left( {{x_0},t} \right),{\omega _i})}}} \frac{{dt}}{{{t^{mn - \alpha  + 1}}}}}\notag\\
	\le &\prod\limits_{i = 1}^m {({{\left\| {{\omega _i}} \right\|}_{{L^{{q_i}}}(B)}}{{\left\| {{\omega _i}^{ - 1}} \right\|}_{{L^{{p_i}^\prime }}(B)}}} )\int_{2r}^\infty  {\prod\limits_{i = 1}^m {{{\left\| {{f_i}} \right\|}_{{L^{{p_i}}}(B\left( {{x_0},t} \right),{\omega _i}^{{p_i}})}}} \frac{{dt}}{{{t^{mn - \alpha  + 1}}}}}\notag\\
	\lesssim & \prod\limits_{i = 1}^m {\left\| {{\omega _i}} \right\|_{{L^{{q_i}}}(B)}^{}} \int_{2r}^\infty  {\prod\limits_{i = 1}^m {({{\left\| {{f_i}} \right\|}_{{L^{{p_i}}}(B\left( {{x_0},t} \right),{\omega _i}^{{p_i}})}}\left\| {{\omega _i}^{ - 1}} \right\|_{{L^{{p_i}^\prime }}(B\left( {{x_0},t} \right))}^{})} {{\left| {B\left( {{x_0},t} \right)} \right|}^{ - m + \frac{\alpha }{n}}}\frac{{dt}}{t}}\notag\\
	\lesssim &\left\| {{u_{\vec \omega }}} \right\|_{{L^q}(B)}^{}\int_{2r}^\infty  {\prod\limits_{i = 1}^m {({{\left\| {{f_i}} \right\|}_{{L^{{p_i}}}(B\left( {{x_0},t} \right),{\omega _i}^{{p_i}})}}\left\| {{\omega _i}} \right\|_{{L^{{q_i}}}(B)}^{ - 1})} \frac{{dt}}{t}},
	\end{align}
	where  the last step holds since we used ${u_{\vec \omega }} \in {A_{\vec P,q}}$ and Lemma \ref{cen4}.
	
	Combining the above estimates, we can easy to get the boundedness of ${{\mathcal I}_{\alpha,m}}$ as follows
	\begin{align*}
	&{\left\| {{\mathcal I}_{\alpha,m}(\vec f)} \right\|_{{M^{q,{\varphi _2}}}\left( {{u}} \right)}}\\
	\le& {\left\| {{\mathcal I}_{\alpha,m}(f_1^0, \ldots ,f_m^0)} \right\|_{{M^{q,{\varphi _2}}}\left( {{u}} \right)}} + \sum\limits_{({\beta _1}, \ldots ,{\beta _m}) \ne 0} {{{\left\| {{\mathcal I}_{\alpha,m}(f_1^{{\beta _1}}, \cdots ,f_m^{{\beta _m}})} \right\|}_{{M^{q,{\varphi _2}}}\left( {{u}} \right)}}}\\
	\lesssim& \mathop {\sup }\limits_{{x_0} \in \rn,r > 0} {\varphi _2}{\left( {{x_0},r} \right)^{ - 1}}\int_{2r}^\infty  {\prod\limits_{i = 1}^m {{{\left\| {{f_i}} \right\|}_{{L^{{p_i}}}(D({x_0},t),\omega _i^{{p_i}})}}\left\| {{\omega _i}} \right\|_{{L^{{q_i}}}(D({x_0},t))}^{ - 1}} \frac{{dt}}{t}}\notag\\
	\lesssim&\mathop {\sup }\limits_{{x_0} \in \rn,r > 0} \prod\limits_{i = 1}^m {\left( {{\varphi _{1i}}{{\left( {{x_0},r} \right)}^{ - 1}}{{\left\| {{f_i}} \right\|}_{{L^{{p_i}}}(D({x_0},r),{\omega _i}^{{p_i}})}}{\omega _i}^{{p_i}}{{\left( {D({x_0},r)} \right)}^{ - \frac{1}{{{p_i}}}}}} \right)}\notag\\
	\le& \prod\limits_{i = 1}^m {{{\left\| {{f_i}} \right\|}_{{M^{{p_i},{\varphi _{1i}}}}({{\omega _i}^{{p_i}}})}}},
	\end{align*}
	where the third inequality holds since we used (\ref{con3}) and (\ref{Gu21}) in Lemma \ref{Gu2}.
\end{proof}

\subsection{Proof of Theorem \ref{CFRI1}}
Without loss of generality, for the sake of simplicity, we only consider the case when $m=2$.
\begin{proof}[Proof]
	For any ball $B=B(x_0,r)$, let $f_i=f^0_i+f^{\infty}_i$, where $f^0_i=f_i\chi_{2B}$, $i=1,\ldots,m$ and $\chi_{2B}$ denotes the characteristic function of $2B$. Then, we have
	\begin{equation*}
	\begin{split}
	&{{u_{\vec \omega }}^q}{\left( B \right)^{-{\frac{1}{q}}}}{\left\| {{{\mathcal I}_{\alpha,m,\prod {\vec b} }}({f_1},{f_2})} \right\|_{{L^q}(B,{u_{\vec \omega }}^q)}}\\
	\le& {{u_{\vec \omega }}^q}{\left( B \right)^{ - \frac{1}{q}}}{\left\| {{{\mathcal I}_{\alpha,m,\prod {\vec b} }}(f_1^0,f_2^0)} \right\|_{{L^q}\left( {B,{u_{\vec \omega }}^q} \right)}} + {{u_{\vec \omega }}^q}{\left( B \right)^{ - \frac{1}{q}}}{\left\| {{{\mathcal I}_{\alpha,m,\prod {\vec b} }}(f_1^0,f_2^\infty )} \right\|_{{L^q}\left( {B,{u_{\vec \omega }}^q} \right)}}\\
	+&{{u_{\vec \omega }}^q}{\left( B \right)^{ - \frac{1}{q}}}{\left\| {{{\mathcal I}_{\alpha,m,\prod {\vec b} }}(f_1^\infty ,f_2^0)} \right\|_{{L^q}\left( {B,{u_{\vec \omega }}^q} \right)}} + {{u_{\vec \omega }}^q}{\left( B \right)^{ - \frac{1}{q}}}{\left\| {{{\mathcal I}_{\alpha,m,\prod {\vec b} }}(f_1^\infty ,f_2^\infty )} \right\|_{{L^q}\left( {B,{u_{\vec \omega }}^q} \right)}}\\
	:=&{J_1}({x_0},r) + {J_2}({x_0},r) + {J_3}({x_0},r) + {J_4}({x_0},r).
	\end{split}
	\end{equation*}
	We first claim that for $i=1, 2, 3, 4$
	\begin{equation}\label{eq3.5}
	{J_i}({x_0},r) \le C \prod\limits_{i = 1}^2 {{{\left\| {{b_i}} \right\|}_{\rm BMO}}} \int_{2r}^\infty  {{{\left( {1 + \log \frac{t}{r}} \right)}^2}\prod\limits_{i = 1}^2 {\left( {{{\left\| {{f_i}} \right\|}_{{L^{{p_i}}}(B({x_0},t), {\omega _i}^{{p_i}})}}\left\| {{\omega _i}} \right\|_{{L^{{q_i}}}(B({x_0},t))}^{ - 1}} \right)\frac{{dt}}{t}} }, 
	\end{equation}
	where $C$ is independent of $r$, $x_0$ and $\vec f$. 
	
	When $\eqref{eq3.5}$ are valid, the proofs of boundedness are similar to the proof ideas in Theorem $\ref{FRM}$, which are given as follows
	\begin{align*}
	&{\left\| {{{{\mathcal I}_{\alpha,m,\prod {\vec b} }}}(\vec f)} \right\|_{{M^{q,{\varphi _2}}}\left( {{u_{\vec \omega }}^q} \right)}}\\
	\le& \mathop {\sup }\limits_{{x_0} \in \rn,r > 0} {\varphi _2}{\left( {{x_0},r} \right)^{ - 1}}\sum {J_i\left( {x_0},r \right)}\\
	\lesssim& \prod\limits_{i = 1}^2 {{{\left\| {{b_i}} \right\|}_{\rm BMO}}} \mathop {\sup }\limits_{{x_0} \in {\rn},r > 0} {\varphi _2}{\left( {{x_0},r} \right)^{ - 1}}\int_r^\infty  {{{\left( {1 + \log \frac{t}{r}} \right)}^2}\prod\limits_{i = 1}^2 {\left( {{{\left\| {{f_i}} \right\|}_{{L^{{p_i}}}(B({x_0},t),{\omega _i}^{{p_i}})}}\left\| {{\omega _i}} \right\|_{{L^{{q_i}}}(B({x_0},t))}^{ - 1}} \right)\frac{{dt}}{t}} }\\
	\lesssim& \prod\limits_{i = 1}^2 {{{\left\| {{b_i}} \right\|}_{\rm BMO}}} \mathop {\sup }\limits_{{x_0} \in {\rn},r > 0} \prod\limits_{i = 1}^2 {\left( {{\varphi _{1i}}{{\left( {{x_0},r} \right)}^{ - 1}}{{\left\| {{f_i}} \right\|}_{{L^{{p_i}}}(B({x_0},r),{\omega _i}^{{p_i}})}}\left\| {{\omega _i}} \right\|_{{L^{{p_i}}}(D({x_0},t))}^{ - 1}} \right)}\\
	\le& \prod\limits_{i = 1}^2 {{{\left\| {{b_i}} \right\|}_{\rm BMO}}} \prod\limits_{i = 1}^2 {{{\left\| {{f_i}} \right\|}_{{M^{{p_i},{\varphi _{1i}}}}({\omega _i}^{p_i})}}},
	\end{align*}
	where the third inequality holds since we used Lemma \ref{Gu3} and $(\ref{con4})$ to ensure (\ref{Gu31}) holds.
	
	From the above proof, we only need to verify the validity of (\ref{eq3.5}). Obviously, we can easily estimate $J_1$ to obtain $(\ref{eq3.5})$ by combining Lemma \ref{LB2} and the process of obtaining the $(\ref{ieq5})$. Note that ${J_{3}}$ is similar to ${J_{2}}$, so we merely consider to estimate ${J_{2}}$ and ${J_{4}}$.
	
	For $z \in B$, we have
	\begin{align}\label{eq3.7}
	& \left| {{{\mathcal I}_{\alpha,m,\prod {\vec b} }}(f_1^0,f_2^\infty )\left( z \right)} \right|\notag\\
	\le& \left| {\left( {{b_1}\left( z \right) - {\mu _1}} \right)\left( {{b_2}\left( z \right) - {\mu _2}} \right){\mathcal I}_{\alpha,m}(f_1^0,f_2^\infty )\left( z \right)} \right| + \left| {\left( {{b_1}\left( z \right) - {\mu _1}} \right){\mathcal I}_{\alpha,m}(f_1^0,\left( {{b_2} - {\mu _2}} \right)f_2^\infty )\left( z \right)} \right|\notag\\
	+& \left| {\left( {{b_2}\left( z \right) - {\mu _2}} \right){\mathcal I}_{\alpha,m}(\left( {{b_1} - {\mu _1}} \right)f_1^0,f_2^\infty )\left( z \right)} \right| + \left| {{\mathcal I}_{\alpha,m}(\left( {{b_1} - {\mu _1}} \right)f_1^0,\left( {{b_2} - {\mu _2}} \right)f_2^\infty )\left( z \right)} \right|\notag\\ 
	:=& {J_{21}}\left( z \right) + {J_{22}}\left( z \right) + {J_{23}}\left( z \right) + {J_{24}}\left( z \right),
	\end{align}
	where ${\mu _j} = {\left( {{b_j}} \right)_B}$.
	
	Using H\"older's inequality and Lemma \ref{Gu1}, we get 
	\begin{align}\label{eq3.8}
	&{\left\| {{J_{21}}} \right\|_{{L^q}\left( {B,{{u_{\vec \omega }}^q}} \right)}}\notag\\
	\le& {\left\| {\left( {{b_1} - {\mu _1}} \right)\left( {{b_2} - {\mu _2}} \right)} \right\|_{{L^q}\left( {B,{{u_{\vec \omega }}^q}} \right)}}{\left\| {{\mathcal I}_{\alpha,m}(f_1^0,f_2^\infty )} \right\|_{{L^\infty }\left( B \right)}}\notag\\
	\lesssim& \prod\limits_{i = 1}^2 {\left( {{{\left\| {{b_i} - {\mu _i}} \right\|}_{{L^{2q}}(B,{{u_{\vec \omega }}^q})}}} \right)} \int_{2r}^\infty  {\prod\limits_{i = 1}^m {{{\left\| {{f_i}} \right\|}_{{L^{{p_i}}}(B({x_0},t), \omega _i^{{p_i}})}}\left\| {{\omega _i}} \right\|_{{L^{{q_i}}}(B({x_0},t))}^{ - 1}} \frac{{dt}}{t}}\notag\\ 
	\lesssim&\left( {\prod\limits_{i = 1}^2 {{{\left\| {{b_i}} \right\|}_{\rm BMO}}} } \right){{{u_{\vec \omega }}^q}}{\left( B \right)^{\frac{1}{q}}} \int_{2r}^\infty  {\prod\limits_{i = 1}^m {{{\left\| {{f_i}} \right\|}_{{L^{{p_i}}}(B({x_0},t),\omega _i^{{p_i}})}}\left\| {{\omega _i}} \right\|_{{L^{{q_i}}}(B({x_0},t))}^{ - 1}} \frac{{dt}}{t}}.
	\end{align}
	For the term ${J_{22}}$, note that inequality $(\ref{ieq6})$ holds for $\mathcal{I}_{\alpha,m}$, and we use the piecewise integration technique again to get
	\begin{align}\label{eq3.9}
	&\left| {{{\mathcal I}_{\alpha,m}}(f_1^0,\left( {{b_2} - {\mu _2}} \right)f_2^\infty )\left( z \right)} \right|\notag\\
	\lesssim&\sum\limits_{j = 1}^\infty  {{{\left( {{2^{j + 1}}r} \right)}^{\alpha  - 2n}}\int_{{2^{j + 1}}B} {\int_{{2^{j + 1}}B} {\left| {{f_1}\left( {{y_1}} \right)\left( {{b_2}\left( {{y_2}} \right) - {\mu _2}} \right){f_2}\left( {{y_2}} \right)} \right|d{y_1}} d{y_2}} }\notag\\ 
	\le&\sum\limits_{j = 1}^\infty  {{{\left( {{2^{j + 1}}r} \right)}^{\alpha  - 2n}}{{\left\| {{f_1}} \right\|}_{{L^{{p_1}}}({2^{j + 1}}B,\omega _1^{{p_1}})}}{{\left\| {{\omega _1}^{ - 1}} \right\|}_{{L^{{p_1}^\prime }}({2^{j + 1}}B)}}{{\left\| {{f_2}} \right\|}_{{L^{{p_2}}}({2^{j + 1}}B,\omega _2^{{p_2}})}}{{\left\| {{b_2} - {\mu _2}} \right\|}_{{L^{{p_2}^\prime }}({2^{j + 1}}B,{\omega _2}^{ - {p_2}^\prime })}}}\notag\\
	\le&\sum\limits_{j = 1}^\infty  {{{\left( {{2^{j + 1}}r} \right)}^{\alpha  - 2n - 1}}\int_{{2^{j + 1}}r}^{{2^{j + 2}}r} {{{\left\| {{\omega _1}^{ - 1}} \right\|}_{{L^{{p_1}^\prime }}(B\left( {{x_0},t} \right))}}{{\left\| {{b_2} - {\mu _2}} \right\|}_{{L^{{p_2}^\prime }}(B\left( {{x_0},t} \right),{\omega _2}^{ - {p_2}^\prime })}}\prod\limits_{i = 1}^2 {{{\left\| {{f_i}} \right\|}_{{L^{{p_i}}}(B\left( {{x_0},t} \right),\omega _i^{{p_i}})}}} dt} }\notag\\
	\lesssim&\int_{2r}^\infty  {{{\left\| {{\omega _1}^{ - 1}} \right\|}_{{L^{{p_1}^\prime }}(B\left( {{x_0},t} \right))}}{{\left\| {{b_2} - {\mu _2}} \right\|}_{{L^{{p_2}^\prime }}(B\left( {{x_0},t} \right),{\omega _2}^{ - {p_2}^\prime })}}\prod\limits_{i = 1}^2 {{{\left\| {{f_i}} \right\|}_{{L^{{p_i}}}(B\left( {{x_0},t} \right),\omega _i^{{p_i}})}}} \frac{{dt}}{{{t^{2n - \alpha  + 1}}}}}\notag\\
	\lesssim& {{{\left\| {{b_2}} \right\|}_{\rm BMO}}} \int_{2r}^\infty  {\left( {1 + \log \frac{t}{r}} \right)\left( {\prod\limits_{i = 1}^2 {\left( {{{\left\| {{f_i}} \right\|}_{{L^{{p_i}}}(B\left( {{x_0},t} \right),\omega _i^{{p_i}})}}{{\left\| {{\omega _i}^{ - 1}} \right\|}_{{L^{{p_i}^\prime }}(B\left( {{x_0},t} \right))}}} \right)} } \right)\frac{{dt}}{{{t^{2n - \alpha  + 1}}}}}\notag\\
	\lesssim&{\left\| {{b_2}} \right\|_{\rm BMO}}\int_{2r}^\infty  {\left( {1 + \log \frac{t}{r}} \right)\left( {\prod\limits_{i = 1}^2 {\left( {{{\left\| {{f_i}} \right\|}_{{L^{{p_i}}}(B\left( {{x_0},t} \right),\omega _i^{{p_i}})}}\left\| {{\omega _i}} \right\|_{{L^{{q_i}}}(B\left( {{x_0},t} \right))}^{ - 1}} \right)} } \right)\frac{{dt}}{t}},
	\end{align}
	where the last step due to the definition of $\vec{\omega} \in {A_{\vec P,q}}$ and Lemma \ref{cen4}.
	Combining with Lemma \ref{Gu1}, it is easy to see that
	\begin{align}\label{eq3.10}
	&{\left\| {{J_{22}}} \right\|_{{L^q}\left( {B,{u_{\vec \omega }}^q} \right)}}\notag\\
	\le&{\left\| {\left( {{b_1} - {\mu _1}} \right)} \right\|_{{L^q}\left( {B,{u_{\vec \omega }}^q} \right)}}{\left\| {{\mathcal I}_{\alpha,m}(f_1^0,\left( {{b_2} - {\mu _2}} \right)f_2^\infty )} \right\|_{{L^\infty }\left( B \right)}}\notag\\
	\lesssim&{u_{\vec \omega }}^q{\left( B \right)^{\frac{1}{q}}}\left( {\prod\limits_{i = 1}^2 {{{\left\| {{b_i}} \right\|}_{BMO}}} } \right)\int_{2r}^\infty  {\left( {1 + \log \frac{t}{r}} \right)\left( {\prod\limits_{i = 1}^2 {\left( {{{\left\| {{f_i}} \right\|}_{{L^{{p_i}}}(B\left( {{x_0},t} \right),\omega _i^{{p_i}})}}\left\| {{\omega _i}} \right\|_{{L^{{q_i}}}(B\left( {{x_0},t} \right))}^{ - 1}} \right)} } \right)\frac{{dt}}{t}}.
	\end{align}
	Similarly, we also have
	\begin{align}\label{eq3.11}
	&{\left\| {{J_{23}}} \right\|_{{L^q}\left( {B,{u_{\vec \omega }}^q} \right)}}\notag\\
	\lesssim&{u_{\vec \omega }}^q{\left( B \right)^{\frac{1}{q}}}\left( {\prod\limits_{i = 1}^2 {{{\left\| {{b_i}} \right\|}_{BMO}}} } \right)\int_{2r}^\infty  {\left( {1 + \log \frac{t}{r}} \right)\left( {\prod\limits_{i = 1}^2 {\left( {{{\left\| {{f_i}} \right\|}_{{L^{{p_i}}}(B\left( {{x_0},t} \right),\omega _i^{{p_i}})}}\left\| {{\omega _i}} \right\|_{{L^{{q_i}}}(D\left( {{x_0},t} \right))}^{ - 1}} \right)} } \right)\frac{{dt}}{t}}.
	\end{align}
	For the term ${J_{24}}$, we use methods similar to getting $(\ref{eq3.9})$ to obtain
	\begin{align}\label{eq3.12}
	&{{J_{24}}}(z)\notag\\ 
	\lesssim&\sum\limits_{j = 1}^\infty  {{{\left( {{2^{j + 1}}r} \right)}^{\alpha  - 2n}}\prod\limits_{i = 1}^2 {\int_{{2^{j + 1}}B} {\left| {\left( {{b_i}\left( {{y_i}} \right) - {\mu _i}} \right){f_i}\left( {{y_i}} \right)} \right|d{y_i}} } }\notag\\ 
	\lesssim& \int_{2r}^\infty  {\prod\limits_{i = 1}^2 {{{\left\| {{f_i}} \right\|}_{{L^{{p_i}}}(B\left( {{x_0},t} \right),\omega _i^{{p_i}})}}} {{\left\| {{b_i} - {\mu _i}} \right\|}_{{L^{{p_i}^\prime }}(B\left( {{x_0},t} \right),{\omega _i}^{ - {p_i}^\prime })}}\frac{{dt}}{{{t^{2n - \alpha  + 1}}}}}\notag\\
	\lesssim&\left( {\prod\limits_{i = 1}^2 {{{\left\| {{b_i}} \right\|}_{\rm BMO}}} } \right)\int_{2r}^\infty  {{{\left( {1 + \log \frac{t}{r}} \right)}^2}\left( {\prod\limits_{i = 1}^2 {\left( {{{\left\| {{f_i}} \right\|}_{{L^{{p_i}}}(B\left( {{x_0},t} \right),\omega _i^{{p_i}})}}\left\| {{\omega _i}} \right\|_{{L^{{q_i}}}(B\left( {{x_0},t} \right))}^{ - 1}} \right)} } \right)\frac{{dt}}{t}}.
	\end{align}
	The above estimates allows us to obtain
	\begin{align}\label{eq3.13}
	&{\left\| {{J_{24}}} \right\|_{{L^q}\left( {B,{u_{\vec \omega }}^q} \right)}}\notag\\
	\lesssim&{u_{\vec \omega }}^q{\left( B \right)^{\frac{1}{q}}}\left( {\prod\limits_{i = 1}^2 {{{\left\| {{b_i}} \right\|}_{\rm BMO}}} } \right)\int_{2r}^\infty  {{{\left( {1 + \log \frac{t}{r}} \right)}^2}\left( {\prod\limits_{i = 1}^2 {\left( {{{\left\| {{f_i}} \right\|}_{{L^{{p_i}}}(B\left( {{x_0},t} \right),\omega _i^{{p_i}})}}\left\| {{\omega _i}} \right\|_{{L^{{q_i}}}(D\left( {{x_0},t} \right))}^{ - 1}} \right)} } \right)\frac{{dt}}{t}}.
	\end{align}
	By using $(\ref{eq3.8})$, $(\ref{eq3.10})$, $(\ref{eq3.11})$ and $(\ref{eq3.13})$, we can obtain the estimates of $J_{2}$:
	\begin{align}\label{eq3.14}
	&{J_2}({x_0},r)\notag\\ 
	\lesssim&\left( {\prod\limits_{i = 1}^2 {{{\left\| {{b_i}} \right\|}_{\rm BMO}}} } \right)\int_{2r}^\infty  {{{\left( {1 + \log \frac{t}{r}} \right)}^2}\left( {\prod\limits_{i = 1}^2 {\left( {{{\left\| {{f_i}} \right\|}_{{L^{{p_i}}}(B\left( {{x_0},t} \right),\omega _i^{{p_i}})}}\left\| {{\omega _i}} \right\|_{{L^{{q_i}}}(D\left( {{x_0},t} \right))}^{ - 1}} \right)} } \right)\frac{{dt}}{t}}.
	\end{align}
	Similar to deal with ${J_{4}}$, we obtain the desired result. Indeed,  we first use a decomposition similar to $(\ref{eq3.7})$, and then we give respectively some estimates similar to the above for each part. Thus, we can obtain (\ref{eq3.5}). This finishes the proof of Theorem $\ref{CFRI1}$.
\end{proof}
\section{Some applications}\label{sec4}
The most basic differential operator is the sub-Laplacian associated with $X_i$ is the second-order partial differential operator given by ${\mathcal L} = \sum\limits_{i = 1}^\gamma  {X_i^2}$, see \cite{Gu8} for more details. From Theorem \ref{thm7}, we can easily obtains an inequality extending the classical Sobolev embedding theorem to generalized weighted Morrey spaces. We only present the results, whose proofs are similar to that in \cite{Gu8}, which we omit here.

\begin{thm}[Sobolev-Stein embedding theorem on generalized weighted Morrey spaces]
	Let $1 < p < \infty$, $\frac{1}{p} - \frac{1}{q} = \frac{1}{n}$, $\omega  \in {A_{p,q}}$ and a group of non-negative measurable functions $({{\varphi }_1},{\varphi _2})$ satisfy the condition:
	\begin{equation}\label{con8}
	{\left[ {{\varphi _1},{\varphi _2}} \right]_B}: = \mathop {\sup }\limits_{x \in {\rn},r > 0} {\varphi _2}{(x,r)^{{\rm{ - }}1}}\int_r^\infty  {\frac{{\mathop {{\rm{essinf}}}\limits_{t < \eta  < \infty } {\varphi _1}(x,\eta ){{\left\| \omega  \right\|}_{{L^{{p_i}}}(D(x,\eta ))}}}}{{{{\left\| \omega  \right\|}_{{L^{{q_i}}}(D(x,t))}}}}\frac{{dt}}{t}}  < \infty.
	\end{equation}
	Then for $f \in C_c^\infty $, we have 
	\begin{align*}
	{\left\| f \right\|_{{M^{q,{\varphi _2}}}({\omega ^q})}} \le {\left\| {{\nabla _{\mathcal L}}(f)} \right\|_{{M^{p,{\varphi _1}}}({\omega ^p})}}
	\end{align*}
\end{thm}

\begin{defn}
	The generalized weighted Besov-Morrey spaces are defined by
	\begin{align*}
	BM_s^{p\theta ,\varphi }(\omega ) = \{ f:{\left\| f \right\|_{BM_s^{p\theta ,\varphi }(\omega )}} = {\left\| f \right\|_{M_{}^{p,\varphi }(\omega )}} + {(\int_{\rn} {\frac{{\left\| {f(x \cdot ) - f( \cdot )} \right\|_{M_{}^{p,\varphi }(\omega )}^\theta }}{{\rho {{(x)}^{n + s\theta }}}}} dx)^{\frac{1}{\theta }}} < \infty \},
	\end{align*}
	where $1 \le p,\theta  \le \infty$, $0 < s < 1$.
\end{defn}

In the following theorem, we prove the boundedness of ${\mathcal I}_\alpha$ on generalized weighted Besov-Morrey spaces, whose proofs are similar to the case that $m=1$, which has been proved in \cite{Gu8}, so we omit it here.

\begin{thm}
	Let $m\geq 2$, $1< p_i<\infty$, $i=1,2,\ldots,m$ with $\frac{1}{p} = \sum\limits_{i = 1}^m {\frac{1}{{{p_i}}}}$,  $\frac{\alpha }{{mn}}{\rm{ = }}\frac{1}{{{p_i}}}{\rm{ - }}\frac{1}{{{q_i}}} \in (0,1)$, $\frac{1}{q} = \sum\limits_{i = 1}^m {\frac{1}{{{q_i}}}}$, $\vec{\omega}=(\omega_1,\ldots,\omega_m) \in {A_{\vec P,q}}$ with ${\omega _i}^{{q_i}} \in {A_\infty }$, $i = 1, \ldots ,m$, and a group of non-negative measurable functions $({{\vec \varphi }_1},{\varphi _2}) = ({\varphi _{11}}, \ldots ,{\varphi _{1m}},{\varphi _2})$ satisfy the condition $(\ref{con3})$
	If $1 \le \theta  \le \infty$ and $0 < s < 1$, then we have 
	$${{\mathcal I}_\alpha } \in B\left( {\prod\limits_{i = 1}^m {BM_s^{{p_i}\theta ,{\varphi _{1i}}}({\omega _i}^{{p_i}})}  \to BM_s^{q\theta ,{\varphi _2}}({u_{\vec \omega }}^q)} \right).$$
\end{thm}

From the above results, we can get the following Sobolev-Stein embedding theorem on generalized weighted Besov-Morrey spaces.

\begin{thm}[Sobolev-Stein embedding theorem on generalized weighted Besov-Morrey spaces]
	Let $1 < p < \infty$, $\frac{1}{p} - \frac{1}{q} = \frac{1}{n}$, $\omega  \in {A_{p,q}}$ and a group of non-negative measurable functions $({{\varphi }_1},{\varphi _2})$ satisfy the condition $(\ref{con8})$.
	If $1 \le \theta  \le \infty$ and $0 < s < 1$, then for $f \in C_c^\infty $, we have 
	\begin{align*}
	{\left\| f \right\|_{BM_s^{q\theta ,{\varphi _2}}({\omega ^q})}} \lesssim {\left\| {{\nabla _{\mathcal L}}(f)} \right\|_{BM_s^{p\theta ,{\varphi _1}}({\omega ^p})}}
	\end{align*}
\end{thm}

The apriori estimates can follows from the above results for the sub-Laplacian $\mathcal L$.
\begin{thm}
	Let $1 < p < \infty$, $\omega  \in {A_{p,q}}$, $1 \le \theta  \le \infty$, $0 < s < 1$, and a group of non-negative measurable functions $({{\varphi }_1},{\varphi _2})$ satisfy the condition $(\ref{con8})$.
	\begin{enumerate}[(i)]
		\item If $\frac{1}{p} - \frac{1}{q} = \frac{2}{n}$, then we have 
		\begin{equation*}
		{\left\| f \right\|_{BM_s^{q\theta ,{\varphi _2}}({\omega ^q})}} \lesssim {\left\| {{\mathcal L}f} \right\|_{BM_s^{p\theta ,{\varphi _1}}({\omega ^p})}}.
		\end{equation*}
		\item If $\frac{1}{p} - \frac{1}{q} = \frac{1}{n}$, then we have	
		\begin{equation*}
		{\left\| {{X_i}f} \right\|_{BM_s^{q\theta ,{\varphi _2}}({\omega ^q})}} \lesssim {\left\| {{\mathcal L}f} \right\|_{BM_s^{p\theta ,{\varphi _1}}({\omega ^p})}}, i = 1, \ldots ,\gamma.
		\end{equation*}
	\end{enumerate}
\end{thm}

\bigskip

\noindent Xi Cen

\smallskip

\noindent {\it Address:} School of Mathematics and Physics, Southwest University of Science and Technology, Mianyang, 621010, P. R. China

\smallskip

\noindent {\it E-mail:} xicenmath@gmail.com

\bigskip

\noindent Qianjun He

\smallskip

\noindent {\it Address:} Qianjun He: School of Applied Science, Beijing Information Science and Technology University, Beijing, 100192, P. R. China.

\smallskip

\noindent {\it E-mail:} qjhe@bistu.edu.cn

\bigskip

\noindent Xiang Li

\smallskip

\noindent {\it Address:} School of Science, Shandong Jianzhu University, Jinan, 250000, P. R. China

\smallskip

\noindent {\it E-mail:} lixiang162@mails.ucas.ac.cn

\bigskip

\noindent Dunyan Yan

\smallskip

\noindent {\it Address:} School of Mathematical Sciences, University of Chinese Academy of Sciences, Beijing, 100049, P. R. China

\smallskip

\noindent {\it E-mail:} ydunyan@ucas.ac.cn
\end{document}